\documentclass[11pt,fleqn]{article}
\usepackage{geometry}                
\usepackage{graphicx}
\usepackage{amsthm,amsmath}
\usepackage{amssymb,dsfont,mathrsfs}

\numberwithin{equation}{section}
\newtheorem{theorem}{Theorem}[section]
\newtheorem{lemma}{Lemma}[section]
\newtheorem{proposition}{Proposition}[section]
\newtheorem{corollary}{Corollary}[section]

\newtheorem{assumption}{Assumption}[section]

\def\ba{\boldsymbol{a}}

\def\bc{\boldsymbol{c}}

\def\bu{\boldsymbol{u}}

\def\bx{\boldsymbol{x}}
\def\by{\boldsymbol{y}}

\def\bJ{\boldsymbol{J}}

\def\bX{\boldsymbol{X}}
\def\bY{\boldsymbol{Y}}

\def\bpi{\boldsymbol{\pi}}
\def\bnu{\boldsymbol{\nu}}
\def\btheta{\boldsymbol{\theta}}

\def\bzero{\mathbf{0}}
\def\bone{\mathbf{1}}

\def\calD{\mathcal{D}}

\def\calS{\mathcal{S}}

\def\scrI{\mathscr{I}}

\def\spr{\mbox{\rm spr}}
\def\cp{\mbox{\rm cp}}

\def\adj{\mbox{\rm adj}}

\title{Asymptotic property of the occupation measures in a two-dimensional skip-free Markov modulated random walk}
\author{Toshihisa Ozawa  \\ 
Faculty of Business Administration, Komazawa University \\
1-23-1 Komazawa, Setagaya-ku, Tokyo 154-8525, Japan \\
E-mail: toshi@komazawa-u.ac.jp
}
\date{}

\begin{document}

\maketitle

\begin{abstract}
We consider a discrete-time two-dimensional process $\{(X_{1,n},X_{2,n})\}$ on $\mathbb{Z}^2$ with a background process $\{J_n\}$ on a finite set $S_0$, where individual processes $\{X_{1,n}\}$ and $\{X_{2,n}\}$ are both skip free. We assume that the joint process $\{\bY_n\}=\{(X_{1,n},X_{2,n},J_n)\}$ is Markovian and that the transition probabilities of the two-dimensional process $\{(X_{1,n},X_{2,n})\}$ vary according to the state of the background process $\{J_n\}$. This modulation is assumed to be space homogeneous. We refer to this process as a two-dimensional skip-free Markov modulate random walk. 
For $\by, \by'\in \mathbb{Z}_+^2\times S_0$, consider the process $\{\bY_n\}_{n\ge 0}$ starting from the state $\by$ and let $\tilde{q}_{\by,\by'}$ be the expected number of visits to the state $\by'$ before the process leaves the nonnegative area $\mathbb{Z}_+^2\times S_0$ for the first time. 
For $\by=(x_1,x_2,j)\in \mathbb{Z}_+^2\times S_0$, the measure $(\tilde{q}_{\by,\by'}; \by'=(x_1',x_2',j')\in \mathbb{Z}_+^2\times S_0)$ is called  an occupation measure. Our main aim is to obtain asymptotic decay rate of the occupation measure as the values of $x_1'$ and $x_2'$ go to infinity in a given direction. We also obtain the convergence domain of the matrix moment generating function of the occupation measures.

\smallskip
{\it Key wards}: Markov modulated random walk, Markov additive process, occupation measure, asymptotic decay rate, moment generating function, convergence domain

\end{abstract}

%
%
\section{Introduction} \label{sec:intro}

%
We consider a discrete-time two-dimensional process $\{(X_{1,n},X_{2,n})\}$ on $\mathbb{Z}^2$, where $\mathbb{Z}$ is the set of all integers, and a background process $\{J_n\}$ on a finite set $S_0=\{1,2,...,s_0\}$, where $s_0$ is the cardinality of $S_0$. We assume that individual processes $\{X_{1,n}\}$ and $\{X_{2,n}\}$ are both skip free, which means that their increments take values in $\{-1,0,1\}$. Furthermore, we assume that the joint process $\{\bY_n\}=\{(X_{1,n},X_{2,n},J_n)\}$ is Markovian and that the transition probabilities of the two-dimensional process $\{(X_{1,n},X_{2,n})\}$ vary according to the state of the background process $\{J_n\}$. This modulation is assumed to be space homogeneous. 
We refer to this process as a two-dimensional skip-free Markov modulate random walk (2d-MMRW for short). The state space of the 2d-MMRW is given by $\mathbb{S}=\mathbb{Z}^2\times S_0$.
The 2d-MMRW is a two-dimensional Markov additive process (2d-MA-process for short) \cite{Miyazawa12}, where $(X_{1,n},X_{2,n})$ is the additive part and $J_n$ the background state. 
A discrete-time two-dimensional quasi-birth-and-death process  \cite{Ozawa13} (2d-QBD process for short) is a 2d-MMRW with reflecting boundaries on the $x_1$ and $x_2$-axes, where the process $(X_{1,n},X_{2,n})$ is the level and $J_n$ the phase. 
Stochastic models arising from various Markovian two-queue models and two-node queueing networks such as two-queue polling models and generalized two-node Jackson networks with Markovian arrival processes and phase-type service processes can be represented as \textit{continuous-time} 2d-QBD processes (see, for example, \cite{Miyazawa15} and \cite{Ozawa13,Ozawa18,Ozawa19a}) and, by using the uniformization technique, they can be deduced to \textit{discrete-time} 2d-QBD processes. 
In that sense, (discrete-time) 2d-QBD processes are more versatile than two-dimensional skip-free reflecting random walks (2d-RRWs for short), which are 2d-QBD processes \textit{without a phase process} and called double QBD processes in \cite{Miyazawa09}. 
It is well known that, in general,  the stationary distribution of a Markov chain can be represented in terms of its stationary probabilities on some boundary faces and its occupation measures. In the case of a 2d-QBD process, such occupation measures are given as those of the corresponding 2d-MMRW. For this reason, we focus on 2d-MMRWs and study their occupation measures, especially, asymptotic properties of the occupation measures. 
Here we briefly explain that the assumption of skip-free is not so restricted. For a given $k>1$, assume that the increments of $X_{1,n}$ and $X_{2,n}$ take values in $\{-k,-(k-1),...,0,1,...,k\}$. For $i\in\{1,2\}$, let ${}^k\!X_{i,n}$ and ${}^k\!M_{i,n}$ be the quotient and remainder of $X_{i,n}$ divided by $k$, respectively, where ${}^k\!X_{i,n}\in\mathbb{Z}$ and $0\le {}^k\!M_{i,n} \le k-1$.  Then, the process $\{({}^k\!X_{1,n},{}^k\!X_{2,n},({}^k\!M_{1,n},{}^k\!M_{2,n},J_n))\}$ becomes a 2d-MMRW \textit{with skip-free jumps}, where $({}^k\!X_{1,n},{}^k\!X_{2,n})$ is the level and $({}^k\!M_{1,n},{}^k\!M_{2,n},J_n)$ the background state. Hence, any 2d-MMRW \textit{with bounded jumps} can be reduced to a 2d-MMRW \textit{with skip-free jumps}. 

%
Let $P=\left(p_{(x_1,x_2,j),(x_1',x_2',j')}; (x_1,x_2,j),(x_1',x_2',j')\in\mathbb{S}\right)$ be the transition probability matrix of the 2d-MMRW $\{\bY_n\}$, where $p_{(x_1,x_2,j)(x_1',x_2',j')}=\mathbb{P}(\bY_{n+1}=(x_1',x_2',j')\,|\,\bY_n=(x_1,x_2,j))$. 
By the property of skip-free, each element of $P$, say $p_{(x_1,x_2,j)(x_1',x_2',j')}$, is nonzero only if $x_1'-x_1\in\{-1,0,1\}$ and $x_2'-x_2\in\{-1,0,1\}$. By the property of space-homogeneity, for $(x_1,x_2),(x_1',x_2')\in\mathbb{Z}^2$, $k,l\in\{-1,0,1\}$ and $j,j'\in S_0$, we have $p_{(x_1,x_2,j),(x_1+k,x_2+l,j')}=p_{(x_1',x_2',j),(x_1'+k,x_2'+l,j')}$. 
Hence, the transition probability matrix $P$ can be represented as a block matrix in terms of only the following $s_0\times s_0$ blocks:
\[
A_{k,l}=\left(p_{(0,0,j)(k,l,j')}; \,j,j'\in S_0\right),\ k,l\in\{-1,0,1\}, 
\]
i.e., for $(x_1,x_2),(x_1',x_2')\in\mathbb{Z}^2$, block $P_{(x_1,x_2)(x_1',x_2')}=(p_{(x_1,x_2,j)(x_1',x_2',j')}; \,j,j'\in S_0)$ is given as 
\begin{equation}
P_{(x_1,x_2)(x_1',x_2')}
= \left\{ \begin{array}{ll}
A_{x_1'-x_1,x_2'-x_2}, & \mbox{if $x_1'-x_1,x_2'-x_2\in\{-1,0,1\}$}, \cr
O, & \mbox{otherwise}, 
\end{array} \right.
\end{equation}
where $O$ is a matrix of 0's whose dimension is determined in context. 
Define a set $\mathbb{S}_+$ as $\mathbb{S}_+=\mathbb{Z}_+^2\times S_0$, where $\mathbb{Z}_+$ is the set of all nonnegative integers, and let $\tau$ be the stopping time at which the 2d-MMRW $\{\bY_n\}$ enters $\mathbb{S}\setminus\mathbb{S}_+$ for the first time, i.e., 
\[
\tau =\inf\{n\ge 0; \bY_n\in\mathbb{S}\setminus\mathbb{S}_+\}.
\]
For $\by=(x_1,x_2,j),\by'=(x_1',x_2',j')\in\mathbb{S}_+$, let $\tilde{q}_{\by,\by'}$ be the expected number of visits to the state $\by'$ before the process $\{\bY_n\}$ starting from the state $\by$ enters $\mathbb{S}\setminus\mathbb{S}_+$ for the first time, i.e.,
\begin{equation}
\tilde{q}_{\by,\by'} = \mathbb{E}\bigg(\sum_{n=0}^{\tau-1} 1\big(\bY_n=\by'\big)\,\Big|\,\bY_0=\by \bigg), 
\label{eq:tildeq_def}
\end{equation}
where $1(\cdot)$ is an indicator function. For $\by\in\mathbb{S}_+$, the measure $(\tilde{q}_{\by,\by'}; \by'=(x_1',x_2',j')\in\mathbb{S}_+)$ is called an occupation measure. 
Note that $\tilde{q}_{\by,\by'}$ is the $(\by,\by')$-element of the fundamental matrix of a truncated substochastic matrix $P_+$ given as $P_+=\left(p_{\by,\by'}; \by,\by'\in\mathbb{S}_+\right)$, i.e., $\tilde{q}_{\by,\by'} = [\tilde{P}_+]_{\by,\by'}$ and 
\[
\tilde{P_+}
=\sum_{k=0}^\infty P_+^k,
\]
where, for example, $P_+^2=\left(p^{(2)}_{\by,\by'}\right)$ is defined by $p^{(2)}_{\by,\by'} = \sum_{\by''\in\calS_+} p_{\by,\by''}\,p_{\by'',\by'}$. $P_+$ governs transitions of $\{\bY_n\}$ on the positive quarter-plane.  
Our main aim is to obtain the asymptotic decay rate of the occupation measure $(\tilde{q}_{\by,\by'}; \by'=(x_1',x_2',j')\in\mathbb{S}_+)$ as the values of $x_1'$ and $x_2'$ go to infinity in a given direction. 
This asymptotic decay rate gives a lower bound for the asymptotic decay rate of the stationary distribution in the corresponding 2d-QBD process in the same direction. Such lower bounds have been obtained for some kinds of multi-dimensional reflected process \textit{without background states}; for example, two-dimensional $0$-partially chains in \cite{Borovkov01}, also see comments on Conjecture 5.1 in \cite{Miyazawa12}. 
With respect to multi-dimensional reflected processes \textit{with background states}, such asymptotic decay rates of the stationary tail distributions in two-dimensional reflected processes have been discussed  in \cite{Miyazawa12,Miyazawa15} by using Markov additive processes and large deviations, but some results seem to be halfway and, in addition, the methods used in those papers are different from ours; we use matrix analytic methods and complex analytic methods. Note that the asymptotic decay rates of the stationary distribution in a 2d-QBD process \textit{in the coordinate directions} have been obtained in \cite{Ozawa13,Ozawa18}.

%
As mentioned above, the 2d-MMRW $\{\bY_n\}=\{(X_{1,n},X_{2,n},J_n)\}$ is a 2d-MA-process, where the set of blocks, $\{ A_{i,j}; i,j\in\{-1,0,1\} \}$, corresponds to the kernel of the 2d-MA-process. 
Let $A_{*,*}(\theta_1,\theta_2)$ be the matrix moment generating function of one-step transition probabilities defined as
\[
A_{*,*}(\theta_1,\theta_2) = \sum_{i,j\in\{-1,0,1\}} e^{i \theta_1+j \theta_2} A_{i,j}.
\]
$A_{*,*}(\theta_1,\theta_2)$ is the Feynman-Kac operator \cite{Ney87} for the 2d-MA-process. 
%
%
For $\bx=(x_1,x_2), \bx'=(x_1',x_2')\in\mathbb{Z}_+^2$, define an $s_0\times s_0$ matrix $N_{\bx,\bx'}$ as $N_{\bx,\bx'}=(\tilde{q}_{(\bx,j),(\bx',j')}; j,j'\in S_0)$ and $N_{\bx}$ as $N_{\bx}=(N_{\bx,\bx'}; \bx'\in\mathbb{Z}_+^2)$. In terms of $N_{\bx,\bx'}$,  $\tilde{P}_+$ is represented as $\tilde{P}_+=(N_{\bx,\bx'}; \bx,\bx'\in\mathbb{Z}_+^2)$. 
For $\bx=(x_1,x_2)\in\mathbb{Z}_+^2$, let $\Phi_{\bx}(\theta_1,\theta_2)$ be the matrix moment generating function of the occupation measures defined as 
\[
\Phi_{\bx}(\theta_1,\theta_2)=\sum_{k_1=0}^\infty \sum_{k_2=0}^\infty e^{k_1 \theta_1+k_2 \theta_2} N_{\bx,(k_1,k_2)}.
\]
For $\bx\in\mathbb{Z}_+^2$, define the convergence domain of $\Phi_{\bx}(\theta_1,\theta_2)$ as
\[
\calD_{\bx} = \mbox{the interior of }\{(\theta_1,\theta_2)\in\mathbb{R}^2 : \Phi_{\bx}(\theta_1,\theta_2)<\infty\}.
\]
Define point sets $\Gamma$ and $\calD$ as 
\begin{align*}
&\Gamma = \left\{(\theta_1,\theta_2)\in\mathbb{R}^2; \spr(A_{*,*}(\theta_1,\theta_2))<1 \right\}, \\
&\calD = \left\{(\theta_1,\theta_2)\in\mathbb{R}^2; \mbox{there exists $(\theta_1',\theta_2')\in\Gamma$ such that $(\theta_1,\theta_2)<(\theta_1',\theta_2')$} \right\}, 
\end{align*}
where $\spr(A_{*,*}(\theta_1,\theta_2))$ is the spectral radius of $A_{*,*}(\theta_1,\theta_2)$. 
In the following sections, we will prove that, for any vector $\bc=(c_1,c_2)$ of positive integers and for every $j,j'\in S_0$, 
\[
\lim_{k\to\infty} \frac{1}{k} \log \tilde{q}_{(0,0,j),(c_1 k,c_2 k,j')} 
= -\sup_{\btheta\in\Gamma}\, \langle \bc,\btheta \rangle,
\]
where $\langle \bc, \btheta \rangle$ is the inner product of vectors $\bc$ and $\btheta$. Furthermore, using this asymptotic property, we will also demonstrate that, for any $\bx\in\mathbb{Z}_+^2$, $\calD_{\bx}$ is given by $\calD$.

%
The rest of the paper is organized as follows.  
In Sect.\ \ref{sec:model}, we present some assumptions and basic properties of the 2d-MMRW.
In Sect.\ \ref{sec:QBDrepresentation}, we introduce three kinds of one-dimensional QBD process with countably many phases and obtain the convergence parameters of the rate matrices in the one-dimensional QBD processes. 
In Sect.\ \ref{sec:convergence_domain}, we consider the matrix generating functions of the occupation measures and demonstrate that the convergence domain of $\Phi_{\bx}(\theta_1,\theta_2)$ contains $\Gamma$. 
Sect.\ \ref{sec:asymptotic} is a main section, where the asymptotic decay rates of the occupation measures and the convergence domain of $\Phi_{\bx}(\theta_1,\theta_2)$ are obtained. 
The paper concludes with remarks on the asymptotic property of 2d-QBD processes in Section \ref{sec:concluding}. 

%
\bigskip
\textit{Notation for matrices.}\quad 
%
%
For a matrix $A$, we denote by $[A]_{i,j}$ the $(i,j)$-element of $A$. 
The convergence parameter of a square matrix $A$ with a finite or countable dimension is denoted by $\cp(A)$, i.e., $\cp(A) = \sup\{r\in\mathbb{R}_+; \sum_{n=0}^\infty r^n A^n<\infty \}$. 
For a finite-dimensional square matrix $A$, we denote by $\spr(A)$ the spectral radius of $A$, which is the maximum modulus of eigenvalue of $A$. If a square matrix $A$ is finite and nonnegative, we have $\spr(A)=\cp(A)^{-1}$. 
The determinant of a square matrix $A$ is denoted by $\det A$ and the adjugate matrix by $\adj\, A$.
%

%
%
\section{Preliminaries} \label{sec:model}

We give some assumptions and propositions, which will be necessary in the following sections.
First, we assume the following condition. 
\begin{assumption} \label{as:P_irreducible}
The 2d-MMRW $\{\bY_n\}=\{(X_{1,n},X_{2,n},J_n)\}$ is irreducible and aperiodic. 
\end{assumption}

Under this assumption, for any $\theta_1,\theta_2\in\mathbb{R}$, $A_{*,*}(\theta_1,\theta_2)$ is also irreducible and aperiodic. 
Denote $A_{*,*}(0,0)$ by $A_{*,*}$ and define the following matrices: for $i,j\in\{-1,0,1\}$, 
\[
A_{*,j}=\sum_{k\in\{-1,0,1\}} A_{k,j},\quad
A_{i,*}=\sum_{k\in\{-1,0,1\}} A_{i,k}. 
\]
$A_{*,*}$ is the transition probability matrix of the background process $\{J_n\}$. 
%
%
Since $A_{*,*}$ is finite and irreducible, it is positive recurrent. Denote by $\bpi_{*,*}$ the stationary distribution of $A_{*,*}$. Define the mean increment vector of the process $\{\bY_n\}$, denoted by $\ba=(a_1,a_2)$, as follows.
\begin{equation}
a_1 = \bpi_{*,*} \left(-A_{-1,*}+A_{1,*} \right) \bone,\quad
a_2 = \bpi_{*,*} \left(-A_{*,-1}+A_{*,1} \right) \bone,
\end{equation}
where $\bone$ is a column vector of 1's whose dimension is determined in context.
With respect to the occupation measures defined in Sect.\ \ref{sec:intro}, the following property holds.
\begin{proposition} \label{pr:finiteness_tildeQ}
If $a_1<0$ or $a_2<0$, then, for any $\by\in\mathbb{S}_+$, the occupation measure $(\tilde{q}_{\by,\by'}; \by'\in\mathbb{S}_+)$ is finite, i.e., 
\begin{equation}
\sum_{\by'\in\mathbb{S}_+} \tilde{q}_{\by,\by'} = \mathbb{E}(\tau\,|\,\bY_0=\by) < \infty, 
\end{equation}
where $\tau$ is the stopping time at which $\{\bY_n\}$ enters $\mathbb{S}\setminus\mathbb{S}_+$ for the first time. 
\end{proposition}
\begin{proof}
Without loss of generality, we assume $a_1<0$. Let $\check{\tau}$ be the stopping time at which $X_{1,n}$ becomes less than 0 for the first time, i.e., $\check{\tau} = \inf\{n\ge 0; X_{1,n}<0 \}$. Since $\{(x_1,x_2,j)\in\mathbb{S}; x_1<0 \} \subset \mathbb{S}\setminus\mathbb{S}_+$, we have $\tau\le \check{\tau}$, and this implies that, for any $\by\in\mathbb{S}_+$, 
\begin{equation}
\mathbb{E}(\tau\,|\,\bY_0=\by)\le\mathbb{E}(\check{\tau}\,|\,\bY_0=\by). 
\end{equation}
Next, we demonstrate that $\mathbb{E}(\check{\tau}\,|\,\bY_0=\by)$ is finite. Consider a one-dimensional QBD process $\{\check{\bY}_n\}=\{(\check{X}_n,\check{J}_n)\}$ on $\mathbb{Z}_+\times S_0$, having $A_{-1,*}$, $A_{0,*}$ and $A_{1,*}$ as the transition probability blocks that govern transitions of the QBD process when $\check{X}_n>0$. We assume the transition probability blocks that govern transitions of the QBD process when $\check{X}_n=0$ are given appropriately. 
Then, $a_1$ is the mean increment of the QBD process when $\check{X}_n>0$ and the assumption of $a_1<0$ implies that the QBD process is positive recurrent. 
Define a stopping time $\check{\tau}^Q$ as $\check{\tau}^Q=\inf\{n\ge 0; \check{X}_n=0\}$. We have, for any $\by=(x_1,x_2,j)\in\mathbb{S}$, 
\begin{equation}
\mathbb{E}(\check{\tau}\,|\,\bY_0=\by) = \mathbb{E}(\check{\tau}^Q\,|\,\check{Y}_0=(x_1+1,j)) < \infty,
\end{equation}
and this completes the proof.
\end{proof}

Hereafter, we assume the following condition. 
\begin{assumption} \label{as:finiteness_tildeQ}
$a_1<0$ or $a_2<0$.
\end{assumption}

%
Let  $\chi(\theta_1,\theta_2)$ be the maximum eigenvalue of $A_{*,*}(\theta_1,\theta_2)$. Since $A_{*,*}(\theta_1,\theta_2)$ is nonnegative, irreducible and aperiodic, $\chi(\theta_1,\theta_2)$ is the Perron-Frobenius eigenvalue of $A_{*,*}(\theta_1,\theta_2)$, i.e., $\chi(\theta_1,\theta_2)=\spr(A_{*,*}(\theta_1,\theta_2))$. 
The modulus of every eigenvalue of $A_{*,*}(\theta_1,\theta_2)$ except $\chi(\theta_1,\theta_2)$ is strictly less than $\spr(A_{*,*}(\theta_1,\theta_2))$. 
We say that a positive function $f(x,y)$ is log-convex in $(x,y)$ if $\log f(x,y)$ is convex in $(x,y)$. A log-convex function is also a convex function. 
With respect to $\chi(\theta_1,\theta_2)$, the following property holds. 
\begin{proposition}[Proposition 3.1 of \cite{Ozawa13}] \label{pr:chiconvex}
$\chi(\theta_1,\theta_2)$ is log-convex and hence convex in $(\theta_1,\theta_2) \in \mathbb{R}^2$. 
\end{proposition}

Let $\bar{\Gamma}$ be the closure of $\Gamma$, i.e., $\bar{\Gamma}=\{ (\theta_1,\theta_2)\in\mathbb{R}^2 : \chi(\theta_1,\theta_2) \le 1\}$.
%
By Proposition \ref{pr:chiconvex}, $\bar{\Gamma}$ is a convex set. 
%
%
Furthermore, the following property holds. 
\begin{proposition}[Lemma 2.2 of \cite{Ozawa18}]  \label{pr:barGamma_bounded}
$\bar{\Gamma}$ is bounded. 
\end{proposition}

%
For $\by=(\bx,j)=(x_1,x_2,j)\in\mathbb{S}_+$, we give an asymptotic inequality for the occupation measure $(\tilde{q}_{\by,\by'}; \by'\in\mathbb{S}_+)$.  
Under Assumption \ref{as:finiteness_tildeQ}, the occupation measure is finite and $(\tilde{q}_{\by,\by'}/E(\tau\,|\,\bY_0=\by); \by'\in\mathbb{S}_+)$ becomes a probability measure. Let $\bY=(\bX,J)=(X_1,X_2,J)$ be a random vector subject to the probability measure, i.e., $P(\bY=\by')= \tilde{q}_{\by,\by'}/E(\tau\,|\,\bY_0=\by)$ for $\by'\in\mathbb{S}_+$.
By Markov's inequality, for $\btheta=(\theta_1,\theta_2)\in\mathbb{R}^2$ and for $\bc=(c_1,c_2)\in\mathbb{R}_+^2$ such that $\bc\ne\bzero$, where $\bzero$ is a vector of 0's whose dimension is determined in context, we have
\begin{align*}
\mathbb{E}(e^{\langle \btheta,\bX \rangle} 1(J=j')) 
&\ge e^{k \langle \btheta,\bc \rangle} P(e^{\langle \btheta,\bX \rangle} 1(J=j') \ge e^{k \langle \btheta,\bc \rangle}) \cr
&= e^{k \langle \btheta,\bc \rangle} P(\langle \btheta,\bX \rangle \ge \langle \btheta,k \bc \rangle, J=j' ) \cr
&\ge e^{k \langle \btheta,\bc \rangle} P(\bX \ge k \bc, J=j' ).
\end{align*}
This implies that, for every $l_1,l_2\in\mathbb{Z}_+$, 
\begin{align}
& [\Phi_{\bx}(\btheta)]_{j,j'} 
\ge e^{k \langle \btheta,\bc \rangle} \sum_{x_1'\ge k c_1}\ \sum_{x_2'\ge k c_2} \tilde{q}_{\by,(x_1',x_2',j')}
\ge e^{k \langle \btheta,\bc \rangle} \tilde{q}_{\by,(k (\lfloor c_1 \rfloor +l_1),k (\lfloor c_2 \rfloor +l_2),j')}, 
\end{align}
where $\lfloor x \rfloor$ is the largest integer less than or equal to $x$. Hence, considering the convergence domain of $\Phi_{\bx}(\btheta)$,we immediately obtain the following basic inequality.
\begin{proposition} \label{pr:limsup_tildeqnn}
For any $\bc=(c_1,c_2)\in\mathbb{Z}_+^2$ such that $\bc\ne\bzero$ and for every $(\bx,j)\in\mathbb{S}_+$, $j'\in S_0$ and $l_1,l_2\in\mathbb{Z}_+$,  
\begin{equation}
 \limsup_{k\to\infty} \frac{1}{k} \log \tilde{q}_{(\bx,j),(k c_1+l_1,k c_2+l_2,j')} \le - \sup_{\btheta\in\calD_{\bx}} \langle \btheta,\bc \rangle. 
 \label{eq:tildeq_upper}
\end{equation}
\end{proposition}


%
%
\section{QBD representations with a countable phase space} \label{sec:QBDrepresentation}

In order to analyze the occupation measures, we introduce three kinds of one-dimensional QBD process with countably many phases. 
Let $\{\bY_n\}=\{(X_{1,n},X_{2,n},J_n)\}$ be a 2d-MMRW and $\tau$ be the stopping time defined in Sect.\ \ref{sec:intro}, i.e.,  $\tau=\inf\{n\ge 0: \bY_n\in\mathbb{S}\setminus\mathbb{S}_+\}$. 
Using $\tau$, we define a process $\{\hat{\bY}_n\}=\{(\hat{X}_{1,n},\hat{X}_{2,n},\hat{J}_n)\}$ as 
\[
\hat{\bY}_n= \bY_{\tau \wedge n},\ n\ge 0, 
\] 
where $x\wedge y= \min\{x,y\}$. The process $\{\hat{\bY}_n\}$ is an absorbing Markov chain on the state space $\mathbb{S}$, where the set of absorbing states is given by $\mathbb{S}\setminus\mathbb{S}_+$. Hereafter, we restrict the state space of $\{\hat{\bY}_n\}$ to $\mathbb{S}_+$, where the transition probability matrix of the process is given by $P_+$. 
We assume the following condition through the paper.
\begin{assumption} \label{as:Q_irreducible}
$P_+$ is irreducible.
\end{assumption}

Under this assumption, $P$ is irreducible regardless of Assumption \ref{as:P_irreducible} and every element of $\tilde{P}_+$ is positive. 
We consider three kinds of QBD representation for $\{\hat{Y}_n\}$: the first is $\{\hat{\bY}_n^{(1)}\}=\{(\hat{X}_{1,n},(\hat{X}_{2,n},\hat{J}_n))\}$, where $\hat{X}_{1,n}$ is the level and $(\hat{X}_{2,n},\hat{J}_n)$ is the phase, the second $\{\hat{\bY}_n^{(2)}\}=\{(\hat{X}_{2,n},(\hat{X}_{1,n},\hat{J}_n))\}$, where $\hat{X}_{2,n}$ is the level and $(\hat{X}_{1,n},\hat{J}_n)$ is the phase, and the third 
\[
\{\hat{\bY}_n^{(1,1)}\}=\{(\min\{\hat{X}_{1,n},\hat{X}_{2,n}\},(\hat{X}_{1,n}-\hat{X}_{2,n},\hat{J}_n))\}, 
\]
where $\min\{\hat{X}_{1,n},\hat{X}_{2,n}\}$ is the level and $(\hat{X}_{1,n}-\hat{X}_{2,n},\hat{J}_n)$ is the phase. 
The state spaces of $\{\hat{\bY}_n^{(1)}\}$ and $\{\hat{\bY}_n^{(2)}\}$ are given by $\mathbb{S}_+$ and that of $\{\hat{\bY}_n^{(1,1)}\}$ by $\mathbb{Z}_+\times\mathbb{Z}\times S_0$.
On the state space $\mathbb{S}_+$, the $k$-th level set of $\{\hat{\bY}_n^{(1,1)}\}$ is given by
\[
\mathbb{L}_k^{(1,1)} = \left\{ (x_1,x_2,j)\in\mathbb{S}_+; \min\{x_1,x_2\}=k \right\}, 
\]
and the level sets satisfy, for $k\ge 0$, $\mathbb{L}_{k+1}=\{(x_1+1,x_2+1,j); (x_1,x_2,j)\in\mathbb{L}_k\}$. It can, therefore, be said that $\{\hat{\bY}_n^{(1,1)}\}$ is a QBD process with level direction vector $(1,1)$. Note that the QBD process $\{\hat{\bY}_n^{(1,1)} \}$ is constructed according to Example 4.2 of \cite{Miyazawa12}.
For $\alpha\in\{(1),(2),(1,1)\}$, the transition probability matrix of $\{\hat{Y}^{\alpha}_n\}$ is given in block tri-diagonal form as
\begin{equation} \label{eq:blockformP}
P^{\alpha} = 
\begin{pmatrix}
A^{\alpha}_0 & A^{\alpha}_1 & & & \cr
A^{\alpha}_{-1} & A^{\alpha}_0 & A^{\alpha}_1 & & \cr
& A^{\alpha}_{-1} & A^{\alpha}_0 & A^{\alpha}_1 & \cr
& & \ddots & \ddots & \ddots 
\end{pmatrix},
\end{equation}
where for $i\in\{-1,0,1\}$, 
\[
A^{(1)}_i = 
\begin{pmatrix}
A_{i,0} & A_{i,1} & & & \cr
A_{i,-1} & A_{i,0} & A_{i,1} & & \cr
& A_{i,-1} & A_{i,0} & A_{i,1}  & \cr
& & \ddots & \ddots & \ddots 
\end{pmatrix},\quad
A^{(2)}_i = 
\begin{pmatrix}
A_{0,i} & A_{1,i} & & & \cr
A_{-1,i} & A_{0,i} & A_{1,i} & & \cr
& A_{-1,i} & A_{0,i} & A_{1,i}  & \cr
& & \ddots & \ddots & \ddots 
\end{pmatrix},
\]
and 
\begin{align*}
& A^{(1,1)}_{-1} =
\begin{pmatrix}
\ddots & \ddots & \ddots & & & & & & \cr
& A_{-1,1} & A_{-1,0} & A_{-1,-1} & & & & & \cr
& & A_{-1,1} & A_{-1,0} & A_{-1,-1} & A_{0,-1} & A_{1,-1} & & \cr
& & & & & A_{-1,-1} & A_{0,-1} & A_{1,-1} & \cr
& & & & & & \ddots & \ddots & \ddots
\end{pmatrix}, \\
& A^{(1,1)}_{0} =
\begin{pmatrix}
\ddots & \ddots & \ddots & & & & & & \cr
& A_{0,1} & A_{0,0} & A_{0,-1} & & & & & \cr
& & A_{0,1} & A_{0,0} & A_{0,-1} & A_{1,-1} & & & \cr
& & & A_{0,1} & A_{0,0} & A_{1,0} & & & \cr
& & & A_{-1,1} & A_{-1,0} & A_{0,0} & A_{1,0} & & \cr
& & & & & A_{-1,0} & A_{0,0} & A_{1,0} & \cr
& & & & & & \ddots & \ddots & \ddots
\end{pmatrix}, \\
& A^{(1,1)}_{1} =
\begin{pmatrix}
\ddots & \ddots & \ddots & & & \cr
& A_{1,1} & A_{1,0} & A_{1,-1} & & \cr
& & A_{1,1} & A_{1,0} & & \cr
& & & A_{1,1} & & & \cr
& & & A_{0,1} & A_{1,1} & & \cr
& & & A_{-1,1} & A_{0,1} & A_{1,1} & \cr
& & & & \ddots & \ddots & \ddots
\end{pmatrix}.
\end{align*}
%

%
For $\alpha\in\{(1),(2),(1,1)\}$, let $R^\alpha$ be the rate matrix generated from the triplet $\{A^\alpha_{-1},A^\alpha_0,A^\alpha_1\}$, which is the minimal nonnegative solution to the matrix quadratic equation:
\begin{equation}
R^\alpha = (R^\alpha)^2 A^\alpha_{-1} + R^\alpha A^\alpha_{0} + A^\alpha_{1}.
\end{equation}
We give $\cp(R^\alpha)$, the convergence parameter of $R^\alpha$.
For $\theta\in\mathbb{R}$, define a matrix function $A^\alpha_*(\theta)$ as
\begin{equation}
A^\alpha_*(\theta) = e^{-\theta} A^\alpha_{-1} + A^\alpha_0 + e^{\theta} A^\alpha_1.  
\end{equation}
Since $P_+$ is irreducible and the number of positive elements of each row and column of $A^\alpha_*(0)$ is finite, we have, by Lemma 2.5 of \cite{Ozawa19}, 
\begin{equation}
\log \cp(R^\alpha) = \sup\{\theta\in\mathbb{R}; \cp(A^\alpha_*(\theta))^{-1}< 1 \}.
\label{eq:cpR_1}
\end{equation}
For $\theta_1,\theta_2\in\mathbb{R}$ and for $i,j\in\{-1,0,1\}$, define matrix functions $A_{*,j}(\theta_1)$ and $A_{i,*}(\theta_2)$ as 
\begin{align*}
& A_{*,j}(\theta_1) = \sum_{k\in\{-1,0,1\}} e^{k \theta_1} A_{k,j},\quad 
A_{i,*}(\theta_2) = \sum_{k\in\{-1,0,1\}} e^{k \theta_2} A_{i,k}. 
\end{align*}
The matrix function $A_{*,*}(\theta_1,\theta_2)$ has been already defined in Sect.\ \ref{sec:intro}. 
%
%
%
Note that the point set $\bar{\Gamma}$ is given as $\bar{\Gamma} = \{(\theta_1,\theta_2)\in\mathbb{R}^2; \cp(A_{*,*}(\theta_1,\theta_2))^{-1} \le 1\}$ and it is a closed convex set. 
For $\alpha\in\{(1),(2),(1,1)\}$, we define three points on the boundary of $\bar{\Gamma}$ as
\begin{align*}
& \bar{\btheta}^{(1)} = (\bar{\theta}_1^{(1)},\bar{\theta}_2^{(1)}) = \arg\max_{(\theta_1,\theta_2) \in\bar{\Gamma}} \theta_1,\quad 
\bar{\btheta}^{(2)} = (\bar{\theta}_1^{(2)},\bar{\theta}_2^{(2)}) = \arg\max_{(\theta_1,\theta_2) \in\bar{\Gamma}} \theta_2, \\
& \bar{\btheta}^{(1,1)} = (\bar{\theta}_1^{(1,1)},\bar{\theta}_2^{(1,1)}) = \arg\max_{(\theta_1,\theta_2) \in\bar{\Gamma}} \theta_1+\theta_2. 
\end{align*}
The matrix function $A^{(1)}_*(\theta_1)$ and $A^{(2)}_*(\theta_2)$ are given in block tri-diagonal form as
\begin{align*}
A^{(1)}_* (\theta_1) = 
\begin{pmatrix}
A_{*,0}(\theta_1) & A_{*,1}(\theta_1) & & & \cr
A_{*,-1}(\theta_1) & A_{*,0}(\theta_1) & A_{*,1}(\theta_1) & & \cr
&A_{*,-1}(\theta_1) & A_{*,0}(\theta_1) & A_{*,1}(\theta_1)  & \cr
& & \ddots & \ddots & \ddots 
\end{pmatrix},\\
A^{(2)}_*(\theta_2) = 
\begin{pmatrix}
A_{0,*}(\theta_2) & A_{1,*}(\theta_2) & & & \cr
A_{-1,*}(\theta_2) & A_{0,*}(\theta_2) & A_{1,*}(\theta_2) & & \cr
& A_{-1,*}(\theta_2) & A_{0,*}(\theta_2) & A_{1,*}(\theta_2)  & \cr
& & \ddots & \ddots & \ddots 
\end{pmatrix}.
\end{align*}
Since $A^{(1)}_* (\theta_1)$ and $A^{(2)}_*(\theta_2)$ are irreducible, we obtain, by Lemma 2.6 of \cite{Ozawa19}, 
\[
\cp(A^{(1)}_* (\theta_1))=\sup_{\theta_2\in\mathbb{R}} \cp(A_{*,*}(\theta_1,\theta_2)),\quad  
\cp(A^{(2)}_*(\theta_2))=\sup_{\theta_1\in\mathbb{R}} \cp(A_{*,*}(\theta_1,\theta_2)),
\]
and this with \eqref{eq:cpR_1} leads us to the following proposition. 
\begin{proposition}
\begin{equation}
\log \cp(R^{(1)}) = \max_{(\theta_1,\theta_2)\in\bar{\Gamma}} \theta_1 = \bar{\theta}_1^{(1)},\quad 
\log \cp(R^{(2)}) = \max_{(\theta_1,\theta_2)\in\bar{\Gamma}} \theta_2 = \bar{\theta}_2^{(2)}. 
\end{equation}
\end{proposition}

For $R^{(1,1)}$, we give an upper bound of $\cp(R^{(1,1)})$. 
$A_*^{(1,1)}(\theta)$ is given in block quintuple-diagonal form as 
{\small 
\begin{align*}
& A_*^{(1,1)}(\theta)\cr
& = \begin{pmatrix}
\ddots & \ddots & \ddots & \ddots & \ddots & & & & \cr
\bar{A}_{*,-2}^{(1,1)}(\theta) & \bar{A}_{*,-1}^{(1,1)}(\theta) & \bar{A}_{*,0}^{(1,1)}(\theta) & \bar{A}_{*,1}^{(1,1)}(\theta) & \bar{A}_{*,2}^{(1,1)}(\theta) & & & & \cr
& \bar{A}_{*,-2}^{(1,1)}(\theta) & \bar{A}_{*,-1}^{(1,1)}(\theta) & \bar{A}_{*,0}^{(1,1)}(\theta) & \bar{A}_{*,1}^{(1,1)}(\theta) & A_{1,-1} & & & \cr
& & \bar{A}_{*,-2}^{(1,1)}(\theta) & \bar{A}_{*,-1}^{(1,1)}(\theta) & A_{*,0}^{(1,1)}(\theta) & A_{*,1}^{(1,1)}(\theta) & A_{*,2}^{(1,1)}(\theta) & & \cr
& & & A_{-1,1} & A_{*,-1}^{(1,1)}(\theta) & A_{*,0}^{(1,1)}(\theta) & A_{*,1}^{(1,1)}(\theta) & A_{*,2}^{(1,1)}(\theta) & \cr
& & & & A_{*,-2}^{(1,1)}(\theta) & A_{*,-1}^{(1,1)}(\theta) & A_{*,0}^{(1,1)}(\theta) & A_{*,1}^{(1,1)}(\theta) & A_{*,2}^{(1,1)}(\theta) \cr
& & & & \ddots & \ddots & \ddots & \ddots & \ddots
\end{pmatrix}, 
\end{align*}
}
where
\begin{align*}
&A_{*,-2}^{(1,1)}(\theta) = e^{\theta} A_{-1,1},\quad
A_{*,-1}^{(1,1)}(\theta) = A_{-1,0}+e^{\theta} A_{0,1},\quad 
A_{*,0}^{(1,1)}(\theta) = e^{-\theta} A_{-1,-1}+A_{0,0}+e^{\theta} A_{1,1}, \\
&A_{*,1}^{(1,1)}(\theta) = e^{-\theta} A_{0,-1}+A_{1,0},\quad 
A_{*,2}^{(1,1)} = e^{-\theta} A_{1,-1}, 
\end{align*}
and
\begin{align*}
&\bar{A}_{*,-2}^{(1,1)}(\theta) = e^{-2\theta} A_{*,-2}^{(1,1)}(\theta),\quad
\bar{A}_{*,-1}^{(1,1)}(\theta) = e^{-\theta} A_{*,-1}^{(1,1)}(\theta),\quad 
\bar{A}_{*,0}^{(1,1)}(\theta) = A_{*,0}^{(1,1)}(\theta), \\
&\bar{A}_{*,1}^{(1,1)}(\theta) = e^{\theta} A_{*,1}^{(1,1)}(\theta),\quad 
\bar{A}_{*,2}^{(1,1)} = e^{2\theta} A_{*,2}^{(1,1)}.
\end{align*}
For $\theta_1,\theta_2\in\mathbb{R}$, define a matrix function $A_{*,*}^{(1,1)}(\theta_1,\theta_2)$ as
\[
A_{*,*}^{(1,1)}(\theta_1,\theta_2)
= \sum_{j=-2}^2 e^{j\theta_2} A_{*,j}^{(1,1)}(\theta_1),  
\]
and consider a partial matrix of $A_*^{(1,1)}(\theta_1)$, denoted by $Q_*^{(1,1)}(\theta_1)$, given as
\[
Q_*^{(1,1)}(\theta_1) 
= \begin{pmatrix}
A_{*,0}^{(1,1)}(\theta_1) & A_{*,1}^{(1,1)}(\theta_1) & A_{*,2}^{(1,1)} & & & & \cr
A_{*,-1}^{(1,1)}(\theta_1) & A_{*,0}^{(1,1)}(\theta_1) & A_{*,1}^{(1,1)}(\theta_1) & A_{*,2}^{(1,1)} & & & \cr
A_{*,-2}^{(1,1)}(\theta_1) & A_{*,-1}^{(1,1)}(\theta_1) & A_{*,0}^{(1,1)}(\theta_1) & A_{*,1}^{(1,1)}(\theta_1) & A_{*,2}^{(1,1)} & & \cr
& A_{*,-2}^{(1,1)}(\theta_1) & A_{*,-1}^{(1,1)}(\theta_1) & A_{*,0}^{(1,1)}(\theta_1) & A_{*,1}^{(1,1)}(\theta_1) & A_{*,2}^{(1,1)}(\theta_1) & \cr
& & \ddots & \ddots & \ddots & \ddots & \ddots
\end{pmatrix}, 
\]
which satisfies $\cp(A_*^{(1,1)}(\theta_1))\le \cp(Q_*^{(1,1)}(\theta_1))$. 
Since $Q_*^{(1,1)}(\theta_1)$ is a block quintuple-diagonal matrix, we have, by Remark 2.5 of \cite{Ozawa19}, $\cp(Q_*^{(1,1)}(\theta_1)) = \sup_{\theta_2\in\mathbb{R}} \cp(A_{*,*}^{(1,1)}(\theta_1,\theta_2))$ and this implies 
\begin{equation}
\log \cp(R^{(1,1)}) \le \sup\{\theta_1\in\mathbb{R}; \cp(A_{*,*}^{(1,1)}(\theta_1,\theta_2))^{-1} < 1\ \mbox{for some}\ \theta_2\in\mathbb{R} \}.
\end{equation}
On the other hand, we have
\begin{align}
&A_{*,*}^{(1,1)}(\theta_1+\theta_2,\theta_1) \cr
&\quad = e^{-2\theta_1} e^{\theta_1+\theta_2} A_{-1,1}
+ e^{-\theta_1} (A_{-1,0}+e^{\theta_1+\theta_2} A_{0,1}) 
+ (e^{-(\theta_1+\theta_2)} A_{-1,-1}+A_{0,0}+e^{\theta_1+\theta_2} A_{1,1}) \cr
&\qquad\quad + e^{\theta_1} (e^{-(\theta_1+\theta_2)} A_{0,-1}+A_{1,0})
+ e^{2\theta_1} e^{-(\theta_1+\theta_2)} A_{1,-1} \cr
&\quad = A_{*,*}(\theta_1,\theta_2), 
\end{align}
and this implies
\begin{align}
&\{\theta_1\in\mathbb{R}; \cp(A_{*,*}^{(1,1)}(\theta_1,\theta_2))^{-1} < 1\ \mbox{for some}\ \theta_2\in\mathbb{R} \} \cr
&\quad =\{\theta_1+\theta_2\in\mathbb{R}; \cp(A_{*,*}^{(1,1)}(\theta_1+\theta_2,\theta_1))^{-1} < 1 \}\cr
&\quad =\{\theta_1+\theta_2\in\mathbb{R}; \cp(A_{*,*}(\theta_1,\theta_2))^{-1} < 1 \}.
\end{align}
Hence, we obtain the following proposition.
\begin{proposition} \label{pr:cpR11_upper}
\begin{equation}
\log\cp(R^{(1,1)}) \le \max_{(\theta_1,\theta_2)\in\bar{\Gamma}} \theta_1+\theta_2 = \bar{\theta}_1^{(1,1)} + \bar{\theta}_2^{(1,1)}. 
\label{eq:cpR11}
\end{equation}
\end{proposition}


%
%
\section{Convergence domain of the moment generating functions}  \label{sec:convergence_domain}

In this section, we prove that, for any $\bx\in\mathbb{Z}_+^2$,  the convergence domain of the matrix moment generating function $\Phi_{\bx}(\theta_1,\theta_2)$ includes the point set $\Gamma$. 
For the purpose, we introduce generating functions of the occupation measures since, in analysis of convergence domain, they are more convenient than the moment generating functions.

%
%
\subsection{Matrix generating functions of the occupation measures}

Recall that $\tilde{P}_+$ is the fundamental matrix of the substochastic matrix $P_+$ and each row of $\tilde{P}_+$ is an occupation measure. Furthermore, $\tilde{P}_+$ is represented in block form as $\tilde{P}_+=(N_{\bx}; \bx\in\mathbb{Z}_+^2)=(N_{\bx,\bx'};\bx,\bx'\in\mathbb{Z}_+^2)$, and $N_{\bx}$ and $N_{\bx,\bx'}$ are given as $N_{\bx}=(N_{\bx,\bx'};\bx'\in\mathbb{Z}_+^2)$ and $N_{\bx,\bx'}=(\tilde{q}_{(\bx,j),(\bx',j')}; j,j'\in S_0)$, respectively. 

For $\bx\in\mathbb{Z}_+^2$, let $\hat{\Phi}_{\bx}(z_1,z_2)$ be the matrix generating function of the occupation measures defined as 
\[
\hat{\Phi}_{\bx}(z_1,z_2)=\sum_{k_1=0}^\infty \sum_{k_2=0}^\infty z_1^{k_1} z_2^{k_2} N_{\bx,(k_1,k_2)}.
\]
%
%
In terms of $\hat{\Phi}_{\bx}(z_1,z_2)$, the matrix moment generating function $\Phi_{\bx}(\theta_1,\theta_2)$ defined in Sect.\ \ref{sec:intro} is given as $\Phi_{\bx}(\theta_1,\theta_2)=\hat{\Phi}_{\bx}(e^{\theta_1},e^{\theta_2})$. The matrix generating function $\hat{\Phi}_{\bx}(z_1,z_2)$ satisfies 
\begin{align}
\hat{\Phi}_{\bx}(z_1,z_2) 
&= N_{\bx,(0,0)} + \hat{\Phi}_{\bx}^{(1)}(z_1) + \hat{\Phi}_{\bx}^{(2)}(z_2) + \hat{\Phi}_{\bx}^{(+)}(z_1,z_2), 
\label{eq:Phix_z1z2}
\end{align}
where 
\begin{align*}
&\hat{\Phi}_{\bx}^{(1)}(z_1) = \sum_{k_1=1}^\infty z_1^{k_1} N_{\bx,(k_1,0)},\quad 
\hat{\Phi}_{\bx}^{(2)}(z_2) = \sum_{k_2=1}^\infty z_2^{k_2} N_{\bx,(0,k_2)},\\
&\hat{\Phi}_{\bx}^{(+)}(z_1,z_2) = \sum_{k_1=1}^\infty \sum_{k_2=1}^\infty z_1^{k_1} z_2^{k_2} N_{\bx,(k_1,k_2)}. 
\end{align*}
%
%
%
Define the following matrix functions:
\begin{align*}
&\hat{C}(z_1,z_2) = \sum_{i,j\in\{-1,0,1\}}  z_1^i z_2^j A_{i,j}, \quad 
\hat{C}_0(z_1,z_2) = \sum_{i,j\in\{0,1\}} z_1^i z_2^j A_{i,j}, \\
&\hat{C}_1(z_1,z_2) = \sum_{i\in\{-1,0,1\}} \sum_{j\in\{0,1\}} z_1^i z_2^j A_{i,j},  \quad
\hat{C}_2(z_1,z_2) = \sum_{i\in\{0,1\}} \sum_{j\in\{-1,0,1\}} z_1^i z_2^j A_{i,j},
\end{align*}
where $A_{*,*}(\theta_1,\theta_2)=\hat{C}(e^{\theta_1},e^{\theta_2})$. 
%
%
Under Assumption \ref{as:finiteness_tildeQ}, the summation of each row of $\tilde{P}_+$ is finite and we obtain $\tilde{P}_+ P_+<\infty$. This leads us to the following recursive formula for $\tilde{P}_+$:
\begin{equation}
\tilde{P}_+=I+\tilde{P}_+P_+. 
\label{eq:tQQ}
\end{equation}
Considering the block structure of $\tilde{P}_+$, we obtain, for $\bx\in\mathbb{Z}_+^2$, the recursive formula for $N_{\bx}$:
\begin{align}
&N_{\bx}=\big( 1(\bx'=\bx) I; \bx'\in\mathbb{Z}_+^2 \big) + N_{\bx} P_+,  
\label{eq:NxQ}
\end{align}
and this leads us to
\begin{align}
\hat{\Phi}_{\bx}(z_1,z_2) 
&= z_1^{x_1} z_2^{x_2} I + N_{\bx,(0,0)} \hat{C}_0(z_1,z_2) + \hat{\Phi}_{\bx}^{(1)}(z_1) \hat{C}_1(z_1,z_2) \cr
&\qquad\qquad\qquad + \hat{\Phi}_{\bx}^{(2)}(z_2) \hat{C}_2(z_1,z_2) + \hat{\Phi}_{\bx}^{(+)}(z_1,z_2) \hat{C}(z_1,z_2).
\label{eq:bvarphi}
\end{align}
Combining this equation with (\ref{eq:Phix_z1z2}), we obtain 
\begin{align}
&\hat{\Phi}_{\bx}^{(+)}(z_1,z_2) (I-\hat{C}(z_1,z_2)) + \hat{\Phi}_{\bx}^{(1)}(z_1) (I-\hat{C}_1(z_1,z_2)) \cr
&\qquad + \hat{\Phi}_{\bx}^{(2)}(z_2) (I-\hat{C}_2(z_1,z_2)) + N_{\bx,(0,0)} (I-\hat{C}_0(z_1,z_2)) - z_1^{x_1} z_2^{x_2} I = O. 
\label{eq:mgf_Phi}
\end{align}
This equation will become a clue for investigating the convergence domain.

%
%
\subsection{Radii of convergence of $\hat{\Phi}_{\bx}^{(1)}(z)$ and $\hat{\Phi}_{\bx}^{(2)}(z)$} \label{sec:Phix_rc}

For $\by=(\bx,j)=(x_1,x_2,j)\in\mathbb{S}_+$, $x_1',x_2'\in\mathbb{Z}_+$ and $j'\in S_0$, define generating functions $\hat{\varphi}_{\by,(x_2',j')}^{(1)}(z)$ and $\hat{\varphi}_{\by,(x_1',j')}^{(2)}(z)$ as 
\[
\hat{\varphi}_{\by,(x_2',j')}^{(1)}(z) = \sum_{k=0}^\infty \tilde{q}_{\by,(k,x_2',j')} z^k,\quad 
\hat{\varphi}_{\by,(x_1',j')}^{(2)}(z) = \sum_{k=0}^\infty \tilde{q}_{\by,(x_1',k,j')} z^k, 
\]
and denote by $r_{\by,(x_2',j')}^{(1)}$ and $r_{\by,(x_1',j')}^{(2)}$ the radii of convergence of them, respectively, i.e., 
\[
r_{\by,(x_2',j')}^{(1)} = \sup\{r\ge 0; \hat{\varphi}_{\by,(x_2',j')}^{(1)}(r)<\infty\},\quad 
r_{\by,(x_1',j')}^{(2)} = \sup\{r\ge 0; \hat{\varphi}_{\by,(x_1',j')}^{(2)}(r)<\infty\}. 
\]
We have 
\[
\hat{\Phi}_{\bx}^{(1)}(z) = \big(\hat{\varphi}_{(\bx,j),(0,j')}^{(1)}(z); j,j'\in S_0 \big),\quad 
\hat{\Phi}_{\bx}^{(2)}(z) = \big(\hat{\varphi}_{(\bx,j),(0,j')}^{(2)}(z); j,j'\in S_0 \big), 
\]
and hence, in order to know the radii of convergence of $\hat{\Phi}_{\bx}^{(1)}(z)$ and $\hat{\Phi}_{\bx}^{(2)}(z)$, it suffices to obtain $r_{(\bx,j),(0,j')}^{(1)}$ and $r_{(\bx,j),(0,j')}^{(2)}$ for $j,j'\in S_0$. 
For the purpose, we present a couple of  propositions. 
\begin{proposition} \label{pr:rc_varphi}
For every $\by,\by'\in\mathbb{S}_+$, $k\in\mathbb{Z}_+$ and $l\in S_0$, we have $r_{\by,(k,l)}^{(1)}=r_{\by',(k,l)}^{(1)}$ and $r_{\by,(k,l)}^{(2)}=r_{\by',(k,l)}^{(2)}$. 
\end{proposition}
\begin{proof}
Recall that, for $\by, (k_1,k_2,l)\in\mathbb{S}_+$, $\tilde{q}_{\by,(k_1,k_2,l)}$ is given by 
\begin{align*}
\tilde{q}_{\by,(k_1,k_2,l)} 
= \mathbb{E}\bigg( \sum_{n=0}^\infty 1(\bY_n=(k_1,k_2,l))1(\tau>n)\,\Big|\,\bY_0=\by \bigg),
\end{align*}
where $\tau$ is the stopping time defined as $\tau=\inf\{n\ge 0: \bY_n\in\mathbb{S}\setminus\mathbb{S}_+\}$. 
For any $\by'\in\mathbb{S}_+$, since $P_+$ is irreducible, there exists $n_0\ge 0$ such that $\mathbb{P}(\bY_{n_0}=\by'\,|\,\bY_0=\by)>0$. Using this $n_0$, we obtain 
\begin{align}
\tilde{q}_{\by,(k_1,k_2,l)} 
&\ge \mathbb{E}\bigg( \sum_{n=n_0}^\infty 1(\bY_n=(k_1,k_2,l))1(\tau>n)\,\Big|\,\bY_{n_0}=\by' \bigg) \mathbb{P}(\bY_{n_0}=\by'\,|\,\bY_0=\by) \cr
&= \tilde{q}_{\by',(k_1,k_2,l)}\,\mathbb{P}(\bY_{n_0}=\by'\,|\,\bY_0=\by), 
\label{eq:inequality_tildeq}
\end{align}
and this implies that $r_{\by,(k_2,l)}^{(1)}\le r_{\by',(k_2,l)}^{(1)}$. 
Exchanging $\by$ with $\by'$, we also obtain $r_{\by',(k_2,l)}^{(1)}\le r_{\by,(k_2,l)}^{(1)}$, and this leads us to $r_{\by,(k_2,l)}^{(1)} = r_{\by',(k_2,l)}^{(1)}$.
The other equation $r_{\by,(k_1,l)}^{(2)} = r_{\by',(k_1,l)}^{(2)}$ can analogously be obtained. 
\end{proof}

Next, we consider the matrix generating functions in matrix geometric form corresponding to $\{\hat{\bY}_n^{(1)}\}$ and $\{\hat{\bY}_n^{(2)}\}$. For $k\in\mathbb{Z}_+$, define matrices $N^{(1)}_{0,k}$ and $N^{(2)}_{0,k}$ and matrix generating functions $\hat{\Phi}^{(1)}_0(z)$ and $\hat{\Phi}^{(2)}_0(z)$ as 
\begin{align*}
&N^{(1)}_{0,k} = \big( N_{(0,x_2),(k,x_2')}; x_2,x_2'\in\mathbb{Z}_+ \big),\quad 
N^{(2)}_{0,k} = \big( N_{(x_1,0),(x_1',k)}; x_1,x_1'\in\mathbb{Z}_+ \big),\\
&\hat{\Phi}^{(1)}_0(z) = \sum_{k=0}^\infty N^{(1)}_{0,k} z^k 
= \Big( \big( \hat{\varphi}_{(0,x_2,j),(x_2',j')}^{(1)}(z); j,j'\in S_0 \big); x_2,x_2'\in\mathbb{Z}_+ \Big),\\
&\hat{\Phi}^{(2)}_0(z) = \sum_{k=0}^\infty N^{(2)}_{0,k} z^k 
= \Big( \big( \hat{\varphi}_{(x_1,0,j),(x_1',j')}^{(2)}(z); j,j'\in S_0 \big); x_1,x_1'\in\mathbb{Z}_+ \Big).
\end{align*}
Further define $N^{(1)}_0$ and $N^{(2)}_0$ as 
\[
N^{(1)}_0 = \big( N^{(1)}_{0,k}; k\in\mathbb{Z}_+ \big),\quad 
N^{(2)}_0 = \big( N^{(2)}_{0,k}; k\in\mathbb{Z}_+ \big). 
\]
From (\ref{eq:tQQ}), we obtain, for $i\in\{1,2\}$,
\begin{equation}
\tilde{P}^{(i)} = I + \tilde{P}^{(i)} P^{(i)}, 
\end{equation}
where $\tilde{P}^{(i)}=\sum_{k=0}^\infty (P^{(i)})^k$ and $P^{(i)}$ is given by \eqref{eq:blockformP}. This leads us to, for $i\in\{1,2\}$, 
\begin{align}
N^{(i)}_0 = \begin{pmatrix} I & O & \cdots \end{pmatrix} + N^{(i)}_0 P^{(i)}, 
\end{align}
and we obtain, for $i\in\{1,2\}$, 
\begin{align}
&N_{0,0}^{(i)} = I + N_{0,0}^{(i)} A_0^{(i)} + N_{0,1}^{(i)} A_{-1}^{(i)},\cr
&N_{0,k}^{(i)} = N_{0,k-1}^{(i)} A_1^{(i)} + N_{0,k}^{(i)} A_0^{(i)} + N_{0,k+1}^{(i)} A_{-1}^{(i)},\ k\ge 1.  
\label{eq:N0n}
\end{align}
The solution to equation (\ref{eq:N0n}) is given as
\begin{align}
&N_{0,k}^{(i)} = N_{0,0}^{(i)} (R^{(i)})^k,\quad 
N_{0,0}^{(i)} = (I-A_0^{(i)}-R^{(i)}A_{-1}^{(i)})^{-1} = \sum_{k=0}^\infty (A_0^{(i)}+R^{(i)}A_{-1}^{(i)})^k, 
\label{eq:N0n_solution}
\end{align}
where we use the fact that $\cp\big(A_0^{(i)}+R^{(i)}A_{-1}^{(i)}\big)<1$ since $\tilde{P}^{(i)}$ is finite. 
From \eqref{eq:N0n_solution} and Fubini's theorem, we obtain 
\begin{align}
\hat{\Phi}_0^{(i)}(z) = \sum_{k=0}^\infty N_{0,0}^{(i)} (R^{(i)})^k z^k = N_{0,0}^{(i)} \sum_{k=0}^\infty (zR^{(i)})^k.
\label{eq:Phiz_R}
\end{align}
This leads us to the following proposition. 
\begin{proposition} \label{pr:some_r_cpR}
There exist some states $(0,x_2,j)$ and $(x_1',0,j')$ in $\mathbb{S}_+$ such that, for every $k\in\mathbb{Z}_+$ and $l\in S_0$, $r_{(0,x_2,j),(k,l)}^{(1)}=\cp(R^{(1)})=e^{\bar{\theta}_1^{(1)}}$ and $r_{(x_1',0,j'),(k,l)}^{(2)}=\cp(R^{(2)})=e^{\bar{\theta}_2^{(2)}}$. 
\end{proposition}

\begin{proof}
Define $R^{(1)}(z)$ as $R^{(1)}(z)=\sum_{k=0}^\infty (zR^{(1)})^k$. Recall that the $((x_2,j),(x_2',j'))$-element of $\hat{\Phi}_0^{(1)}(z)$ is $\hat{\varphi}_{(0,x_2,j),(x_2',j')}^{(1)}(z)$. 
Furthermore, $P_+$ is irreducible and every element of $N_{0,0}^{(1)}$ is positive. Hence, by \eqref{eq:Phiz_R}, we have 
\begin{equation}
\hat{\Phi}_0^{(1)}(z)<\infty\Rightarrow R^{(1)}(z)<\infty, 
\end{equation}
and this implies that, for every $x_2,x_2'\in\mathbb{Z}_+$ and every $j,j'\in S_0$, $r_{(0,x_2,j),(x_2',j')}^{(1)}\le\cp(R^{(1)})$. 
On the other hand, we have $R^{(1)}=A_1^{(1)} N_{0,0}^{(1)}$ and, by Fubini's theorem,  
\begin{align}
z A_1^{(1)} \hat{\Phi}_0^{(1)}(z) 
&= \sum_{k=0}^\infty z A_1^{(1)} N_{0,0}^{(1)} (zR^{(1)})^k 
= \sum_{k=0}^\infty (zR^{(1)})^{k+1} 
\le R^{(1)}(z). 
\end{align}
Since $A_1^{(1)}$ is a block tri-diagonal matrix and the size of each block is finite, the number of positive elements in each row of $A_1^{(1)}$ is finite. Since $P_+$ is irreducible, at least one element of $A_1^{(1)}$, say the $((x_0,j_0),(x_2,j))$-element, is positive. Hence, we have, for every $k\in\mathbb{Z}_+$ and $l\in S_0$, 
\begin{equation}
R^{(1)}(z)<\infty \Rightarrow \hat{\varphi}_{(0,x_2,j),(k,l)}^{(1)}(z)<\infty,  
\end{equation}
and this implies that $\cp(R^{(1)})\le r_{(0,x_2,j),(k,l)}^{(1)}$. As a result, we obtain, for some $x_2\in\mathbb{Z}_+$ and $j\in S_0$ and for every $k\in\mathbb{Z}_+$ and $l\in S_0$, $r_{(0,x_2,j),(k,l)}^{(1)}=\cp(R^{(1)})$. 

Analogously, we obtain, for some $x_1\in\mathbb{Z}_+$ and $j\in S_0$ and for every $k\in\mathbb{Z}_+$ and $l\in S_0$, $r_{(x_1,0,j),(k,l)}^{(2)}=\cp(R^{(2)})$. 
\end{proof}

Propositions \ref{pr:rc_varphi} and \ref{pr:some_r_cpR} lead us to the following lemma. 
\begin{lemma} \label{le:r_cpR}
For every $\by\in\mathbb{S}_+$ and for every $x_1',x_2'\in\mathbb{Z}_+$ and $j'\in S_0$, we have $r_{\by,(x_1',j')}^{(1)}=\cp(R^{(1)})=e^{\bar{\theta}_1^{(1)}}$ and $r_{\by,(x_2',j')}^{(2)}=\cp(R^{(2)})=e^{\bar{\theta}_2^{(2)}}$. Hence, for every $\bx\in\mathbb{Z}_+^2$, the radius of convergence of $\hat{\Phi}_{\bx}^{(1)}(z)$ is given by $e^{\bar{\theta}_1^{(1)}}$ and that of $\hat{\Phi}_{\bx}^{(2)}(z)$ by $e^{\bar{\theta}_2^{(2)}}$.
\end{lemma}

%
%
\subsection{Radius of convergence of another matrix generating function}

In the previous subsection, we defined the matrix generating functions of the occupation measures in matrix geometric form corresponding to $\{\hat{\bY}_n^{(1)}\}$ and $\{\hat{\bY}_n^{(2)}\}$. In this subsection, we consider that corresponding to $\{\hat{\bY}_n^{(1,1)}\}$. 
For $k\in\mathbb{Z}_+$, define $N_{0,k}^{(1,1)}$ as
\begin{align*}
N_{0,k}^{(1,1)} 
&= \left( N_{(x_1,x_2),(x_1',x_2')}; \min\{x_1,x_2\}=0,\,x_1-x_2\in\mathbb{Z},\,\min\{x_1',x_2'\}=k,\,x_1'-x_2'\in\mathbb{Z} \right) \cr
&= \begin{pmatrix}
& \vdots & \vdots & \vdots & \cr
\cdots & N_{(0,1),(k,k+1)} & N_{(0,1),(k,k)} & N_{(0,1),(k+1,k)} & \cdots \cr
\cdots & N_{(0,0),(k,k+1)} & N_{(0,0),(k,k)} & N_{(0,0),(k+1,k)} & \cdots \cr
\cdots & N_{(1,0),(k,k+1)} & N_{(1,0),(k,k)} & N_{(1,0),(k+1,k)} & \cdots \cr
& \vdots & \vdots & \vdots & 
\end{pmatrix}. 
\end{align*}
Then, in a manner similar to that used for deriving \eqref{eq:N0n_solution}, we obtain 
\begin{equation}
N_{0,k}^{(1,1)} = N_{0,0}^{(1,1)} (R^{(1,1)})^k,\quad 
N_{0,0}^{(1,1)} = (I-A_0^{(1,1)}-R^{(1,1)}A_{-1}^{(1,1)})^{-1}. 
\label{eq:Nn11}
\end{equation}
Define a matrix generating function $\hat{\Phi}_0^{(1,1)}(z)$ as
\[
\hat{\Phi}_0^{(1,1)}(z) = \sum_{k=0}^\infty N_{0,k}^{(1,1)} z^k = N_{0,0}^{(1,1)} \sum_{k=0}^\infty (zR^{(1,1)})^k.
\]
For $k,k'\in\mathbb{Z}$ and $j,j'\in S_0$, let $\hat{\varphi}_{(k,j),(k',j')}^{(1,1)}(z)$ be the $((k,j),(k',j'))$-element of $\hat{\Phi}_0^{(1,1)}(z)$ and denote by $r_{(k,j),(k',j')}^{(1,1)}$ the radius of convergence of $\hat{\varphi}_{(k,j),(k',j')}^{(1,1)}(z)$. 
In terms of $\tilde{q}_{\by,\by'}$, $\hat{\varphi}_{(k,j),(k',j')}^{(1,1)}(z)$ is given as
\begin{equation}
\hat{\varphi}_{(k,j),(k',j')}^{(1,1)}(z) 
= \sum_{l=0}^\infty \tilde{q}_{(k\vee 0,-(k\wedge 0),j),(l+(k'\vee 0),l-(k'\wedge 0),j')} z^l,  
\end{equation}
where $x\vee y=\max\{x,y\}$ and $x\wedge y=\min\{x,y\}$. We have the following property of $r_{(k,j),(k',j')}^{(1,1)}$, which corresponds to Lemma \ref{le:r_cpR}. 
\begin{lemma} \label{le:r_cpR11}
For every $k,k'\in\mathbb{Z}$ and $j,j'\in S_0$, we have $r_{(k,j),(k',j')}^{(1,1)}=\cp(R^{(1,1)})$. Hence, the radius of convergence of $\hat{\Phi}_0^{(1,1)}(z)$ is given by $\cp(R^{(1,1)})$.
\end{lemma}

Since this lemma can be proved in a manner similar to that used for proving Propositions \ref{pr:rc_varphi} and \ref{pr:some_r_cpR}, we omit the proof. 
%
%
%

%
%
\subsection{Convergence domain of $\Phi_{\bx}(\theta_1,\theta_2)$}

Recall that, for $\bx\in\mathbb{Z}_+^2$, the convergence domain of the matrix moment generating function $\Phi_{\bx}(\theta_1,\theta_2)$ is given as $\calD_{\bx} = \mbox{the interior of }\{(\theta_1,\theta_2)\in\mathbb{R}^2 : \Phi_{\bx}(\theta_1,\theta_2)<\infty\}$. This domain does not depend on $\bx$. 
\begin{proposition} \label{pr:Dx=Dxp}
For every $\bx,\bx'\in\mathbb{Z}_+^2$, $\calD_{\bx}=\calD_{\bx'}$. 
\end{proposition}
\begin{proof}
For every $\bx,\bx'\in\mathbb{Z}_+^2$ and $j\in S_0$, since $P_+$ is irreducible, there exists $n_0\ge 0$ such that $\mathbb{P}(\bY_{n_0}=(\bx',j)\,|\,\bY_0=(\bx,j))>0$. Using this $n_0$ and inequality (\ref{eq:inequality_tildeq}), we obtain, for every $\bx,\bx'\in\mathbb{Z}_+^2$ and $j,j'\in S_0$, 
\begin{align}
[\Phi_{\bx}(\theta_1,\theta_2)]_{j,j'}
&= \sum_{k_1=0}^\infty \sum_{k_2=0}^\infty \tilde{q}_{(\bx,j),(k_1,k_2,j')} e^{\theta_1 k_1+\theta_2 k_2} \cr
&\ge \sum_{k_1=0}^\infty \sum_{k_2=0}^\infty \tilde{q}_{(\bx',j),(k_1,k_2,j')} e^{\theta_1 k_1+\theta_2 k_2}\, \mathbb{P}(\bY_{n_0}=(\bx',j)\,|\,\bY_0=(\bx,j)) \cr
&= [\Phi_{\bx'}(\theta_1,\theta_2)]_{j,j'}\, \mathbb{P}(\bY_{n_0}=(\bx',j)\,|\,\bY_0=(\bx,j)), 
\end{align}
and this implies $\calD_{\bx}\subset \calD_{\bx'}$. 
Exchanging $\bx$ with $\bx'$, we obtain $\calD_{\bx'}\subset \calD_{\bx}$, and this completes the proof. 
\end{proof}

A relation between the point sets $\Gamma$ and $\calD_x$ is given as follows.
\begin{lemma} \label{le:domain_Gamma}
For every $\bx\in\mathbb{Z}_+^2$, $\Gamma\subset\calD_{\bx}$ and hence, $\calD\subset\calD_{\bx}$. 
\end{lemma}

We use complex analytic methods for proving this lemma. 
Letting $z$ and $w$ be complex variables, we define complex matrix functions $\hat{\Phi}_{\bx}(z,w)$, $\hat{\Phi}_{\bx}^{(1)}(z)$, $\hat{\Phi}_{\bx}^{(2)}(w)$, $\hat{\Phi}_{\bx}^{(+)}(z,w)$, $\hat{C}(z,w)$, $\hat{C}_0(z,w)$, $\hat{C}_1(z,w)$ and $\hat{C}_2(z,w)$ in the same manner as that used in the case of real variable. They satisfy equation \eqref{eq:Phix_z1z2}.
For $r>0$, denote by $\Delta_r$, $\bar{\Delta}_r$ and $\partial\Delta_r$ the open disk, closed disk and circle of center $0$ and radius $r$ on the complex plane, respectively. 
Since every element of $\tilde{P}_+$ is nonnegative, we immediately obtain, by Lemma \ref{le:r_cpR}, the following proposition.
\begin{proposition} \label{pr:Phi_analytic}
Under Assumption \ref{as:Q_irreducible}, for every $\bx\in\mathbb{Z}_+^2$, $\hat{\Phi}_{\bx}^{(1)}(z)$ is element-wise analytic in $\Delta_{e^{\bar{\theta}_1^{(1)}}}$ and $\hat{\Phi}_{\bx}^{(2)}(w)$ is element-wise analytic in $\Delta_{e^{\bar{\theta}_2^{(2)}}}$.
\end{proposition}

%
Recall that the matrix generating function $\hat{\Phi}_{\bx}^{(+)}(z,w)$ is the function of two variables defined by the power series 
\begin{equation}
\hat{\Phi}_{\bx}^{(+)}(z,w)=\sum_{k_1=1}^\infty \sum_{k_2=1}^\infty N_{\bx,(k_1,k_2)} z^{k_1} w^{k_2}, 
\label{eq:varphip_series}
\end{equation}
which is also the Taylor series for the function $\hat{\Phi}_{\bx}^{(+)}(z,w)$ at $(z,w)=(0,0)$. 
Since $\Phi_{\bx}^{(+)}(\theta_1,\theta_2)=\hat{\Phi}_{\bx}^{(+)}(e^{\theta_1},e^{\theta_2})$, we use power series \eqref{eq:varphip_series} for proving Lemma \ref{le:domain_Gamma}. 
Under Assumption \ref{as:finiteness_tildeQ}, since $\hat{\Phi}_{\bx}^{(+)}(1,1)<\infty$, power series (\ref{eq:varphip_series}) converges absolutely on $\bar{\Delta}_1\times\bar{\Delta}_1$ and $\hat{\Phi}_{\bx}^{(+)}(z,w)$ is analytic in $\Delta_1\times\Delta_1$.
Lemma \ref{le:domain_Gamma} asserts that, for every $(p,q)\in\Gamma$, power series (\ref{eq:varphip_series}) converges absolutely in $\bar{\Delta}_{e^p}\times\bar{\Delta}_{e^q}$, and to prove it, it suffices to show that, every $(p,q)\in\Gamma$, power series (\ref{eq:varphip_series}) converges at $(z,w)=(e^p,e^q)$ since any coefficient in the power series is nonnegative. 
Define a matrix function $H(z,w)$ as $H(z,w) = F(z,w)/g(z,w)$, where 
\begin{align}
F(z,w) & = z^{s_0} w^{s_0} \bigl( \hat{\Phi}_{\bx}^{(1)}(z) (\hat{C}_1(z,w)-I) + \hat{\Phi}_{\bx}^{(2)}(w) (\hat{C}_2(z,w)-I) \cr 
&\qquad\qquad\qquad\qquad + N_{\bx,(0,0)} (\hat{C}_0(z,w)-I) + z^{x_1} w^{x_2} \bigr)\,\adj(I-\hat{C}(z,w)), \label{eq:Fzw} \\
g(z,w) &=z^{s_0} w^{s_0}\det(I-\hat{C}(z,w)). 
\label{eq:gzw}
\end{align}
Note that the inverse of $I-\hat{C}(z,w)$ is given by $\adj(I-\hat{C}(z,w))/\det(I-\hat{C}(z,w))$. The function $g(z,w)$ is analytic in $\mathbb{C}^2$ and, by Proposition \ref{pr:Phi_analytic}, each element of $F(z,w)$ is analytic in $\Delta_{e^{\bar{\theta}_1^{(1)}}}\times\Delta_{e^{\bar{\theta}_2^{(2)}}}$.  
For $(z,w)\in\bar{\Delta}_1\times\bar{\Delta}_1$, since power series (\ref{eq:varphip_series}) is absolutely convergent, equation (\ref{eq:mgf_Phi}) holds and we have $\hat{\Phi}_{\bx}^{(+)}(z,w)=H(z,w)$. Hence, by the identity theorem, we see that $H(z,w)$ is an analytic extension of $\hat{\Phi}_{\bx}^{(+)}(z,w)$. Hereafter, we denote the analytic extension by the same notation $\hat{\Phi}_{\bx}^{(+)}(z,w)$.
In the proof of Lemma \ref{le:domain_Gamma}, we use the following proposition. 
\begin{proposition}[Proposition 4.1 of \cite{Ozawa18}] \label{pr:sprCzw}
Under Assumption \ref{as:P_irreducible}, for $z,w\in\mathbb{C}$ such that $z\ne 0$ and $w\ne 0$, if $|z|\ne z$ or $|w|\ne w$, then $\spr(\hat{C}(z,w))<\spr(\hat{C}(|z|,|w|)$. 
\end{proposition}

%
%
\begin{figure}[htbp]
\begin{center}
\includegraphics[width=8.5cm,trim=150 250 150 30]{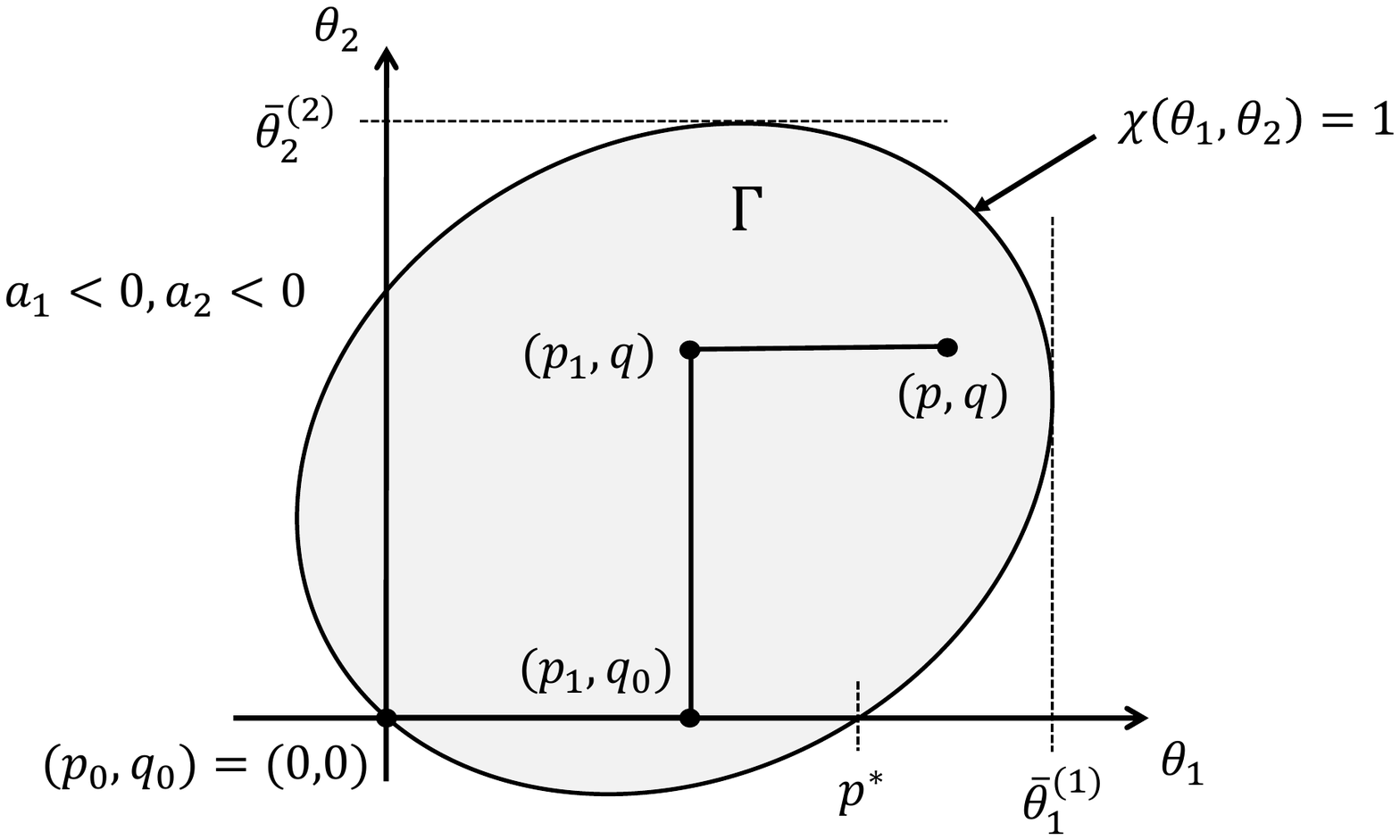} 
\caption{Convergence domain of $\hat{\Phi}_{\bx}^{(+)}(e^{\theta_1},e^{\theta_2})$}
\label{fig:conv_domain_1}
\end{center}
\end{figure}
\begin{proof}[Proof of Lemma \ref{le:domain_Gamma}]
By Lemma \ref{le:r_cpR}, the radius of convergence of $\hat{\Phi}_{\bx}^{(1)}(z)$ is $e^{\bar{\theta}_1^{(1)}}$ and that of $\hat{\Phi}_{\bx}^{(2)}(w)$ is $e^{\bar{\theta}_2^{(2)}}$. Since we have, for every $p,q\in\Gamma$, $p<\bar{\theta}_1^{(1)}$ and $q<\bar{\theta}_2^{(2)}$, both $\hat{\Phi}_{\bx}^{(1)}(z)$ and $\hat{\Phi}_{\bx}^{(2)}(w)$ converge absolutely at every $(z,w)\in\Delta_{e^p}\times\Delta_{e^q}$. 
Hence, from (\ref{eq:Phix_z1z2}), we see that, in order to prove the lemma, it suffices to show that, for every $(p,q)\in\Gamma$, power series (\ref{eq:varphip_series}) converges at $(z,w)=(e^p,e^q)$. 
We demonstrate it. 

Under Assumption \ref{as:finiteness_tildeQ}, either $a_1$ or $a_2$ is negative. Here, we assume $a_1<0$; the proof for the case where $a_2<0$ is analogous. By Lemma 2.3 of \cite{Ozawa18}, we have $\chi_{\theta_1}(0,0)=a_1<0$, where $\chi_{\theta_1}(\theta_1,\theta_2)=(\partial/\partial\,\theta_1)\,\chi(\theta_1,\theta_2)$, and this implies that $\Gamma$ includes open interval $\{(\theta_1,0)\in\mathbb{R}^2; 0<\theta_1<p^* \}$ on the complex plain, where $p^*$ is the value of $\theta_1$ that satisfies $\chi(p^*,0)=1$ (see Fig.\ \ref{fig:conv_domain_1}). 
Let $(p,q)$ be an arbitrary point in $\Gamma$ and consider a path on $\Gamma\cup\{(0,0)\}$ connecting different points $(p_0,q_0)=(0,0)$, $(p_1,q_0)$, $(p_1,q)$ and $(p,q)$ by lines in this order, where we assume $1<p_1<p^*$ (see Fig.\ \ref{fig:conv_domain_1}). It is always possible since the closure of $\Gamma$, $\bar{\Gamma}$, is a closed convex set. 

First, we consider a matrix function of one variable given by $\hat{\Phi}_{\bx}^{(+)}(z,1)$. The Taylor series for $\hat{\Phi}_{\bx}^{(+)}(z,1)$ at $z=0$ is given by power series (\ref{eq:varphip_series}), where $w$ is set at $1$, and $\hat{\Phi}_{\bx}^{(+)}(z,1)$ is identical to $H(z,1)$ on a domain where $H(z,1)$ is well defined. 
Let $\varepsilon$ be a sufficiently small positive number. Since $g(z,1)$ and each element of $F(z,1)$ are  analytic as a function of one variable in $\Delta_{e^{p_1}+\varepsilon}$, $H(z,1)$ is meromorphic as a function of one variable in that domain. 
We have, for $z\in\Delta_{e^{p_1}+\varepsilon}\setminus\bar{\Delta}_1$, $\spr(\hat{C}(z,1))\le\spr(\hat{C}(|z|,1))<1$, 
and by Proposition \ref{pr:sprCzw}, we have, for $z\in\partial\Delta_{1}\setminus\{1\}$, $\spr(\hat{C}(z,1))<\spr(\hat{C}(|z|,1))=1$. 
Hence, $g(z,1)\ne 0$ for any $z\in\Delta_{e^{p_1}+\varepsilon}\setminus(\Delta_1\cup\{1\})$, and each element of $H(z,1)$ is analytic on $\Delta_{e^{p_1}+\varepsilon}\setminus(\Delta_1\cup\{1\})$. 
Furthermore, we have $\hat{\Phi}_{\bx}^{(+)}(1,1)=H(1,1)<\infty$, and this implies that the point $z=1$ is not a pole of any element of $H(z,1)$; hence, it is a removable singularity. 
From this and the fact that $\hat{\Phi}_{\bx}^{(+)}(z,1)$ is analytic in $\Delta_{1}$, we see that $\hat{\Phi}_{\bx}^{(+)}(z,1)$ is analytic in $\Delta_{e^{p_1}+\varepsilon}$. This implies that the radius of convergence of the Taylor series for $\hat{\Phi}_{\bx}^{(+)}(z,1)$ at $z=0$ is greater than $e^{p_1}$, and power series (\ref{eq:varphip_series}) converges at $(z,w)=(e^{p_1},1)$.

Next, we consider a matrix function of one variable given by $\hat{\Phi}_{\bx}^{(+)}(e^{p_1},w)$. By the fact obtained above, the Taylor series for $\hat{\Phi}_{\bx}^{(+)}(e^{p_1},w)$ at $w=0$ is given by power series (\ref{eq:varphip_series}), where $z$ is set at $e^{p_1}$, and $\hat{\Phi}_{\bx}^{(+)}(e^{p_1},w)$ is analytic as a function of one variable in $\Delta_{1}$.
Furthermore, we know that $\hat{\Phi}_{\bx}^{(+)}(e^{p_1},w)$ is identical to $H(e^{p_1},w)$ on a domain where $H(e^{p_1},w)$ is well defined. 
If $q\le 0$, then it is obvious that power series (\ref{eq:varphip_series}) converges at $(z,w)=(e^{p_1},e^q)$. Therefore, we assume $q>0$. 
Let $\varepsilon$ be a sufficiently small positive number. For $w\in\Delta_{e^q+\varepsilon}\setminus\bar{\Delta}_{1-\varepsilon}$, we have $\spr(\hat{C}(e^{p_1},w))\le\spr(\hat{C}(e^{p_1},|w|))<1$. Hence, for the same reason used in the case of $H(z,1)$, we see that $H(e^{p_1},w)$ is analytic in $\Delta_{e^q+\varepsilon}\setminus\bar{\Delta}_{1-\varepsilon}$, and this implies that   $\hat{\Phi}_{\bx}^{(+)}(e^{p_1},w)$ is analytic in $\Delta_{e^q+\varepsilon}$.
Hence, the radius of convergence of the Taylor series for $\hat{\Phi}_{\bx}^{(+)}(e^{p_1},w)$ at $w=0$ is greater than $e^q$, and power series (\ref{eq:varphip_series}) converges at $(z,w)=(e^{p_1},e^q)$.
Applying a similar procedure to the matrix function $\hat{\Phi}_{\bx}^{(+)}(z,e^q)$, we see that power series (\ref{eq:varphip_series}) converges at $(z,w)=(e^p,e^q)$, and this completes the proof. 
\end{proof}

In the following section, we will prove that $\calD=\calD_{\bx}$ holds for every $\bx\in\mathbb{Z}_+^2$ (see Corollary \ref{co:domain_Phix}).

%
%
%
\section{Asymptotics of the occupation measures} \label{sec:asymptotic}

%
%
\subsection{Asymptotic decay rate in an arbitrary direction}

By Lemma \ref{le:r_cpR} and the Cauchy-Hadamard theorem, we have, for every $x_1,x_2, x_1',x_2'\in\mathbb{Z}_+$ and $j,j'\in S_0$, 
\begin{align}
& \limsup_{k\to\infty} \frac{1}{k} \log \tilde{q}_{(0,x_2,j),(k,x_2',j')} = - \sup_{(\theta_1,\theta_2)\in\Gamma} \theta_1 = - \sup_{\btheta\in\Gamma} \langle (1,0), \btheta \rangle, \label{eq:rate_tildeq10} \\
& \limsup_{k\to\infty} \frac{1}{k} \log  \tilde{q}_{(x_1,0,j),(x_1',k,j')} = - \sup_{(\theta_1,\theta_2)\in\Gamma} \theta_2 = - \sup_{\btheta\in\Gamma} \langle (0,1), \btheta \rangle.  \label{eq:rate_tildeq01}
\end{align}
Furthermore, by \eqref{eq:N0n_solution} and Corollary 2.1 of \cite{Ozawa19}, ``$\limsup$" in equations \eqref{eq:rate_tildeq10} and \eqref{eq:rate_tildeq01} can be replaced with ``$\lim$".
The following results are inferred from these equations. 
\begin{theorem} \label{th:asymptotic_any_direction}
For any vector $\bc=(c_1,c_2)$ of positive integers, for every $x_1,x_2\in\mathbb{Z}_+$ such that $x_1=0$ or $x_2=0$, for every $l_1,l_2\in\mathbb{Z}_+$ such that $l_1=0$ or $l_2=0$ and for every $j,j'\in S_0$, 
\begin{align}
\lim_{k\to\infty} \frac{1}{k} \log \tilde{q}_{(x_1,x_2,j),(c_1 k+l_1,c_2 k+l_2,j')} 
&= - \sup_{\btheta\in\Gamma}\, \langle \bc, \btheta \rangle. 
\label{eq:limsup_tildeq}
\end{align}
\end{theorem}

%
In order to prove this theorem, we introduce another representation of the 2d-MMRW $\{\bY_n\}=\{(X_{1,n},X_{2,n},J_n)\}$. 
Let $\bc=(c_1,c_2)$ be a vector of positive integers. For $i\in\{1,2\}$, denote by ${}^{\bc}\!X_{i,n}$ and ${}^{\bc}\!M_{i,n}$ the quotient and remainder of $X_{i,n}$ divided by $c_i$, respectively, i.e., 
\[
X_{1,n}=c_1 {}^{\bc}\!X_{1,n}+{}^{\bc}\!M_{1,n},\quad
X_{2,n}=c_2 {}^{\bc}\!X_{2,n}+{}^{\bc}\!M_{2,n},
\]
where $0\le {}^{\bc}\!M_{1,n}\le c_1-1$ and $0\le {}^{\bc}\!M_{2,n}\le c_2-1$. 
Define a process $\{{}^{\bc}\bY_n\}$ as $\{{}^{\bc}\bY_n\}=\{({}^{\bc}\!X_{1,n},{}^{\bc}\!X_{2,n},{}^{\bc}\!\bJ_n)\}$, where ${}^{\bc}\!\bJ_n=({}^{\bc}\!M_{1,n},{}^{\bc}\!M_{2,n},J_n)$. The process $\{{}^{\bc}\bY_n\}$ is a 2d-MMRW with background process $\{{}^{\bc}\!\bJ_n\}$ and its state space is given by $\mathbb{Z}^2\times(\mathbb{Z}_{0,c_1-1}\times\mathbb{Z}_{0,c_2-1}\times S_0)$, where, for $k,l\in\mathbb{Z}$ such that $k\le l$, we denote by $\mathbb{Z}_{k,l}$ the set of integers from $k$ through $l$, i.e., $\mathbb{Z}_{k,l}=\{k,k+1,...,l\}$. 
The transition probability matrix of $\{{}^{\bc}\bY_n\}$, denoted by ${}^{\bc}\!P$, has a double-tridiagonal block structure like $P$. Denote by  ${}^{\bc}\!A_{i,j},\,i,j\in\{-1,0,1\}$, the nonzero blocks of ${}^{\bc}\!P$. 
For a positive integer $k$ and for $ k_1,k_2\in\{1,2,...,k\}$, define a $k\times k$ matrix $E_{(k_1,k_2)}^{[k]}$ as
\[
[E_{(k_1,k_2)}^{[k]}]_{i,j} = 
\left\{ \begin{array}{ll}
1, & \mbox{$i=k_1$ and $j=k_2$}, \cr
0, & \mbox{otherwise}.
\end{array} \right.
\]
Define the following $c_1\times c_1$ block matrices: for $j\in\{-1,0,1\}$, 
\begin{align*}
B_{-1,j} = E_{(1,c_1)}^{[c_1]}\otimes A_{-1,j},\ 
B_{0,j} = 
\begin{pmatrix}
A_{0,j} & A_{1,j} & & & \cr
A_{-1,j}  & A_{0,j} & A_{1,j} & & \cr
& \ddots & \ddots & \ddots & \cr
& & A_{-1,j}  & A_{0,j} & A_{1,j}  \cr
& & & A_{-1,j} & A_{0,j} 
\end{pmatrix},\ 
B_{1,j} = E_{(c_1,1)}^{[c_1]}\otimes A_{1,j}, 
\end{align*}
where $\otimes$ is the Kronecker product operator. 
Each block ${}^{\bc}\!A_{i,j}$ is a $c_1c_2\times c_1c_2$ block matrix and they are given as follows: for  $i\in\{-1,0,1\}$, 
\begin{align*}
{}^{\bc}\!A_{i,-1} = E_{(1,c_2)}^{[c_2]}\otimes B_{i,-1},\  
{}^{\bc}\!A_{i,0} = 
\begin{pmatrix}
B_{i,0} & B_{i,1} & & & \cr
B_{i,-1}  & B_{i,0} & B_{i,1} & & \cr
& \ddots & \ddots & \ddots & \cr
& & B_{i,-1}  & B_{i,0} & B_{i,1}  \cr
& & & B_{i,-1} & B_{i,0} 
\end{pmatrix},\ 
{}^{\bc}\!A_{i,1} = E_{(c_2,1)}^{[c_2]}\otimes B_{i,1}.
\end{align*}
Define a matrix function ${}^{\bc}\!A_{*,*}(\theta_1,\theta_2)$ as
\[
{}^{\bc}\!A_{*,*}(\theta_1,\theta_2) =\sum_{i,j\in\{-1,0,1\}} e^{i\theta_1+j\theta_2}\, {}^{\bc}\!A_{i,j}.
\]
The following relation holds between $A_{*,*}(\theta_1,\theta_2)$ and ${}^{\bc}\!A_{*,*}(\theta_1,\theta_2)$.
\begin{proposition} \label{pr:QBDcp}
For any vector $\bc=(c_1,c_2)$ of positive integers, we have 
\begin{equation}
\cp(A_{*,*}(\theta_1,\theta_2)) = \cp({}^{\bc}\!A_{*,*}(c_1 \theta_1,c_2 \theta_2)). 
\end{equation}
\end{proposition}
Since the proof of this proposition is elementary, we give it in Appendix \ref{sec:proof_QBDcp}.

%
\begin{proof}[Proof of Theorem \ref{th:asymptotic_any_direction}]
From Proposition \ref{pr:limsup_tildeqnn} and Lemma \ref{le:domain_Gamma}, we immediately obtain, for every $x_1,x_2\in\mathbb{Z}_+$, $l_1,l_2\in\mathbb{Z}_+$ and $j,j'\in S_0$,  
\begin{equation}
\limsup_{k\to\infty} \frac{1}{k} \log \tilde{q}_{(x_1,x_2,j),(c_1 k+l_1,c_2 k+l_2,j')} \le -\sup_{\btheta\in\Gamma}\,\langle \bc,\btheta \rangle.
\end{equation}

In order to obtain the lower bounds, define a process $\{{}^{\bc}\hat{\bY}_n\}=\{({}^{\bc}\!\hat{X}_{1,n},{}^{\bc}\!\hat{X}_{2,n},{}^{\bc}\!\hat{\bJ}_n)\}$ as ${}^{\bc}\hat{\bY}_n={}^{\bc}\bY_{\tau\wedge n}$, where $\tau$ is the stopping time given as $\tau=\inf\{n\ge 0; \bY_n\in\mathbb{S}\setminus\mathbb{S}_+\}$. 
According to $\{\hat{\bY}_n^{(1,1)}\}$ considered in Sect.\ \ref{sec:QBDrepresentation}, define a process $\{{}^{\bc}\hat{\bY}_n^{(1,1)}\}$ as 
\[
{}^{\bc}\hat{\bY}_n^{(1,1)} = 
(\min\{{}^{\bc}\!\hat{X}_{1,n},{}^{\bc}\!\hat{X}_{2,n}\},({}^{\bc}\!\hat{X}_{1,n}-{}^{\bc}\!\hat{X}_{2,n},{}^{\bc}\!\hat{\bJ}_n)). 
\]
The process $\{{}^{\bc}\hat{\bY}_n^{(1,1)}\}$ is a QBD process with countably many phases, where  $\min\{{}^{\bc}\!\hat{X}_{1,n},{}^{\bc}\!\hat{X}_{2,n}\}$ is the level and $({}^{\bc}\!\hat{X}_{1,n}-{}^{\bc}\!\hat{X}_{2,n},{}^{\bc}\!\hat{\bJ}_n)$ is the phase. 
Let ${}^{\bc}\!R^{(1,1)}$ be the rate matrix of $\{{}^{\bc}\hat{\bY}_n^{(1,1)}\}$.  By Propositions \ref{pr:cpR11_upper} and \ref{pr:QBDcp}, we have
\begin{align}
\log \cp({}^{\bc}\!R^{(1,1)}) 
&\le \sup\{\theta_1+\theta_2\in\mathbb{R}; \cp({}^{\bc}\!A_{*,*}(\theta_1,\theta_2))^{-1}\le 1\}  \cr
&= \sup\{c_1\theta_1+c_2\theta_2\in\mathbb{R}; \cp({}^{\bc}\!A_{*,*}(c_1 \theta_1,c_2 \theta_2))^{-1}\le 1\}  \cr
&= \sup\{c_1\theta_1+c_2\theta_2\in\mathbb{R}; \cp(A_{*,*}(\theta_1,\theta_2))^{-1}\le 1\}.
 \label{eq:cpcR_upper}
\end{align}
For some $(i'',l_1'',l_2'',j'')\in\mathbb{Z}\times\mathbb{Z}_{0,c_1-1}\times\mathbb{Z}_{0,c_2-1}\times S_0$ and for any $(i',l_1',l_2',j')\in\mathbb{Z}\times\mathbb{Z}_{0,c_1-1}\times\mathbb{Z}_{0,c_2-1}\times S_0$, we have, by Corollary 2.1 of \cite{Ozawa19}, 
\begin{equation}
\lim_{k\to\infty} \left( [({}^{\bc}\!R^{(1,1)})^k]_{(i'',l_1'',l_2'',j''),(i',l_1',l_2',j')} \right)^{\frac{1}{k}} = \cp({}^{\bc}\!R^{(1,1)})^{-1}.
\end{equation}
If $i'\ge 0$, then the state ${}^{\bc}\hat{\bY}_n^{(1,1)}=(k,i',(l_1',l_2',j'))$ corresponds to the state $\bY_n=(c_1 k+c_1 i'+l_1',c_2 k+l_2',j')$. Hence, from \eqref{eq:Nn11}, setting $l=c_1 i'+l_1'$ and $l_2'=0$, we obtain, for every $x_1,x_2\in\mathbb{Z}_+$ such that $x_1=0$ or $x_2=0$ and for every $j\in S_0$,
\begin{equation}
\tilde{q}_{(x_1,x_2,j),(c_1 k+l,c_2 k,j')} \ge \tilde{q}_{(x_1,x_2,j),(c_1 i''+l_1'',l_2'',j'')} [({}^{\bc}\!R^{(1,1)})^k]_{(i'',l_1'',l_2'',j''),(i',l_1',0,j')}. 
\label{eq:tildeq_under1}
\end{equation}
Analogously, if $i'< 0$, then setting $l_1'=0$ and $l=-c_2 i'+l_2'$, we obtain,  for every $x_1,x_2\in\mathbb{Z}_+$ such that $x_1=0$ or $x_2=0$ and for every $j\in S_0$,
\begin{equation}
\tilde{q}_{(x_1,x_2,j),(c_1 k,c_2 k+l,j')} \ge \tilde{q}_{(x_1,x_2,j),(l_1'',c_2 i''+l_2'',j'')} [({}^{\bc}\!R^{(1,1)})^k]_{(i'',l_1'',l_2'',j''),(i',0,l_2',j')}. 
\label{eq:tildeq_under2}
\end{equation}
From \eqref{eq:tildeq_under1}, \eqref{eq:tildeq_under2} and \eqref{eq:cpcR_upper}, we obtain the desired lower bound as follows: for every $x_1,x_2\in\mathbb{Z}_+$ such that $x_1=0$ or $x_2=0$, for every $l_1,l_2\in\mathbb{Z}_+$ such that $l_1=0$ or $l_2=0$ and for every $j,j'\in S_0$, 
\begin{equation}
\liminf_{k\to\infty} \frac{1}{k} \log \tilde{q}_{(x_1,x_2,j),(c_1 k+l_1,c_2 k+l_2,j')} \ge -\cp({}^{\bc}\!R^{(1,1)})) \ge -\sup_{\btheta\in\Gamma}\,\langle \bc,\btheta \rangle.
\end{equation}
This completes the proof. 
\end{proof}

%
%
%
%
%
%
%

%
From Lemma \ref{le:domain_Gamma} and Theorem \ref{th:asymptotic_any_direction}, we obtain the following property of the convergence domain.
\begin{corollary} \label{co:domain_Phix}
For every $\bx\in\mathbb{Z}_+^2$, $\calD_{\bx}=\calD$.
\end{corollary}

In order to prove the corollary, we introduce some notation. 
Recall that, by Propositions \ref{pr:chiconvex} and \ref{pr:barGamma_bounded}, the point set $\bar{\Gamma}=\{(\theta_1,\theta_2)\in\mathbb{R}^2; \chi(\theta_1,\theta_2)=\spr(A_{*,*}(\theta_1,\theta_2))\le 1\}$ is a closed convex set and bounded. 
%
%
For $i\in\{1,2\}$, define a point $(\underline{\theta}_1^{(i)},\underline{\theta}_2^{(i)})$ as $(\underline{\theta}_1^{(i)},\underline{\theta}_2^{(i)}) = \arg\min_{(\theta_1,\theta_2)\in\bar{\Gamma}} \theta_i$. We have already defined $(\bar{\theta}_1^{(1)},\bar{\theta}_2^{(1)})$ and $(\bar{\theta}_1^{(2)},\bar{\theta}_2^{(2)})$. 
For a given $\theta_1\in(\underline{\theta}_1^{(1)}, \bar{\theta}_1^{(1)})$, equation $\chi(\theta_1,\theta_2)=1$ has just two different real solutions: $\theta_2=\underline{\zeta}_2(\theta_1),\,\bar{\zeta}_2(\theta_1)$, where $\underline{\zeta}_2(\theta_1)<\bar{\zeta}_2(\theta_1)$.  
%
%
The function $\bar{\zeta}_2(\theta_1)$ is monotonically decreasing in $(\bar{\theta}_1^{(2)},\bar{\theta}_1^{(1)})$. 
Further define the following domains:
\begin{align*}
&\Gamma^{(1)} =  \{(\theta_1,\theta_2)\in\mathbb{R}^2; \theta_1<\bar{\theta}_1^{(1)} \} = \{(\theta_1,\theta_2)\in\mathbb{R}^2; \cp(A_*^{(1)}(\theta_1))^{-1}<1 \}, \\
&\Gamma^{(2)} =  \{(\theta_1,\theta_2)\in\mathbb{R}^2; \theta_2<\bar{\theta}_2^{(2)} \} = \{(\theta_1,\theta_2)\in\mathbb{R}^2; \cp(A_*^{(2)}(\theta_2))^{-1}<1 \}. 
\end{align*}

\begin{figure}[htbp]
\begin{center}
\includegraphics[width=17cm,trim=20 290 20 20]{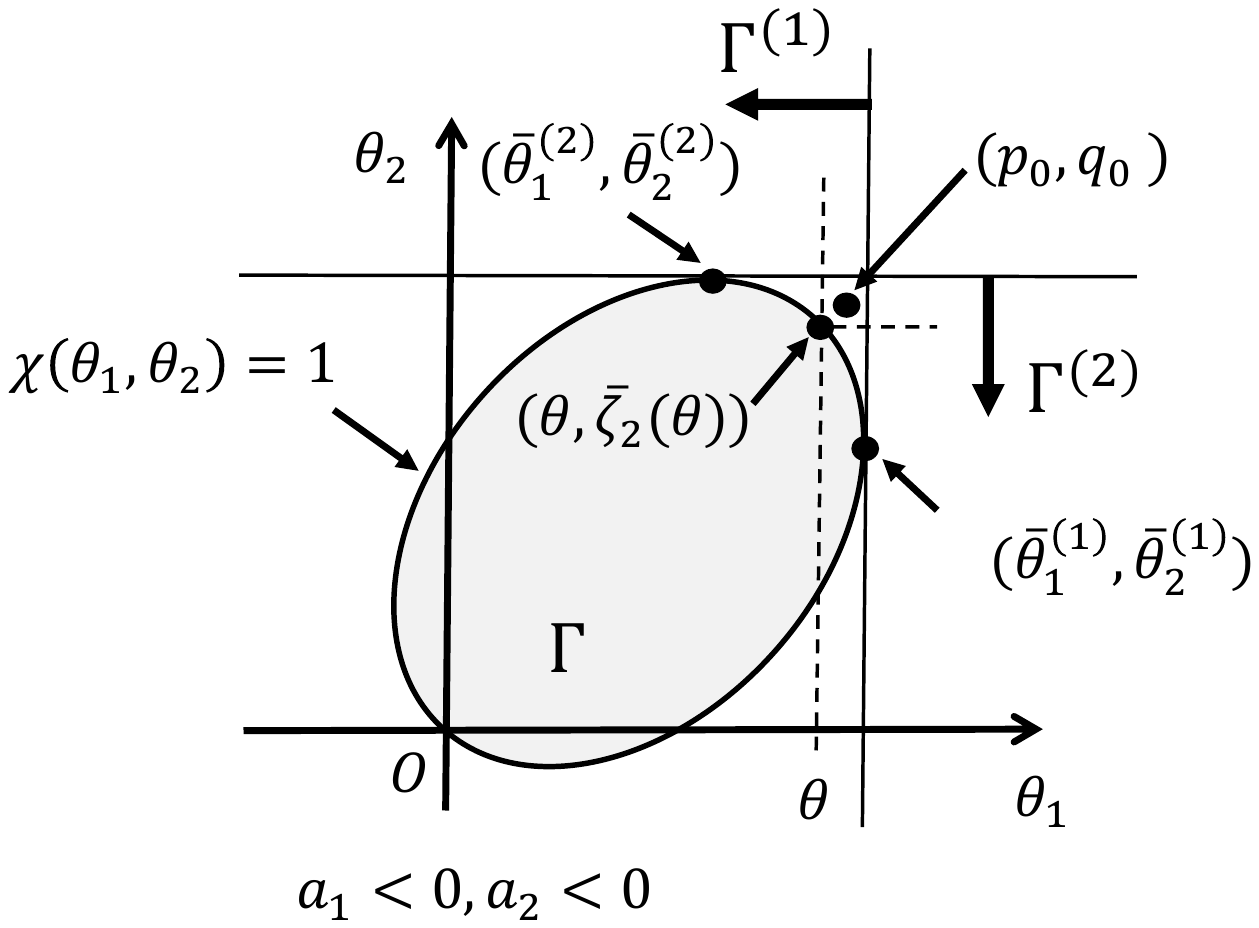} 
\caption{Convergence domain of $\Phi_{\bx}(\theta_1,\theta_2)$}
\label{fig:conv_domain_2}
\end{center}
\end{figure}

\begin{proof}[Proof of Corollary \ref{co:domain_Phix}]
We prove $\calD_{\bzero}=\calD$, where $\bzero=(0,0)$. By Proposition \ref{pr:Dx=Dxp}, this implies $\calD_{\bx}=\calD$ for every $\bx\in\mathbb{Z}_+^2$. 
Regarding $\hat{\Phi}_{\bzero}^{(1)}(e^{\theta_1})$ and $\hat{\Phi}_{\bzero}^{(2)}(e^{\theta_2})$ as functions of two variables $\theta_1$ and $\theta_2$, we see, by Lemma \ref{le:r_cpR}, that the convergence domain of $\hat{\Phi}_{\bzero}^{(1)}(e^{\theta_1})$ is given by $\Gamma^{(1)}$ and that of $\hat{\Phi}_{\bzero}^{(1)}(e^{\theta_2})$ by $\Gamma^{(2)}$. From (\ref{eq:Phix_z1z2}), we, therefore, obtain $
\calD_{\bzero}\subset \Gamma^{(1)} \cap \Gamma^{(2)}$. 
On the other hand, by Lemma \ref{le:domain_Gamma}, we have $\calD \subset \calD_{\bzero}$. Hence, in order to prove $\calD_{\bzero}=\calD$, it suffices to demonstrate that $((\Gamma^{(1)} \cap \Gamma^{(2)})\setminus\calD) \cap \calD_{\bzero} = \emptyset$ (see Fig.\ \ref{fig:conv_domain_2}). 

%
Let $\bc=(c_1,c_2)$ be a vector of positive integers. Define a point ${}^{\bc}\btheta=({}^{\bc}\theta_1,{}^{\bc}\theta_2)$ as 
\[
({}^{\bc}\theta_1,{}^{\bc}\theta_2) = \arg\max_{\btheta\in\bar{\Gamma}}\, \langle \bc,\btheta \rangle.
\]
The point ${}^{\bc}\btheta$ is represented as ${}^{\bc}\btheta=({}^{\bc}\theta_1,\bar{\zeta}_2({}^{\bc}\theta_1))$ and satisfies 
\[
\frac{d}{d \theta} (c_1 \theta + c_2 \bar{\zeta}_2(\theta))\Big|_{\theta={}^{\bc}\theta_1} = c_1 + c_2 \bar{\zeta}'_2({}^{\bc}\theta_1) = 0. 
\]
Hence, we obtain $\bar{\zeta}'_2({}^{\bc}\theta_1) = -c_1/c_2$. 
Since $\bar{\zeta}'_2(\theta)$ monotonically decreases from $0$ to $-\infty$ when $\theta$ increases from $\bar{\theta}_1^{(2)}$ to $\bar{\theta}_1^{(1)}$, we see that the point set $\mathbb{D}_0=\{({}^{\bc}\theta_1,{}^{\bc}\theta_2); \bc=(c_1,c_2)\in\mathbb{Z}_+^2,\,c_1>0,\,c_2>0\}$ is dense in the curve $\mathbb{D}=\{(\theta,\bar{\zeta}_2(\theta)); \theta\in[\bar{\theta}_1^{(2)},\bar{\theta}_1^{(1)}] \}$. 
For $j,j'\in S_0$, define a moment generating function $\varphi_{\bc}(\theta_1,\theta_2)$ as 
\begin{equation}
\varphi_{\bc}(\theta_1,\theta_2) = \sum_{k=0}^\infty e^{(c_1 \theta_1 +c_2 \theta_2)k}\, \tilde{q}_{(0,0,j),(c_1 k,c_2 k,j')}.
\label{eq:varphic}
\end{equation}
By Theorem \ref{th:asymptotic_any_direction} and the Cauchy-Hadamard theorem, we see that the radius of convergence of the power series in the right hand side of \eqref{eq:varphic} is $e^{c_1 {}^{\bc}\theta_1+c_2 {}^{\bc}\theta_2}$ and this implies that $\varphi_{\bc}(\theta_1,\theta_2)$ diverges if $\theta_1>{}^{\bc}\theta_1$ and $\theta_2>{}^{\bc}\theta_2$. 
From the definition of $\varphi_{\bc}(\theta_1,\theta_2)$, we obtain 
\begin{equation}
\varphi_{\bc}(\theta_1,\theta_2) \le \sum_{k_1=0}^\infty \sum_{k_2=0}^\infty e^{\theta_1 k_1+\theta_2 k_2}\, \tilde{q}_{(0,0,j),(k_1,k_2,j')} 
= [\Phi_{\bzero}(\theta_1,\theta_2)]_{j,j'}.
\label{eq:varphic_Phi}
\end{equation}
%
%
Here, we suppose $((\Gamma^{(1)} \cap \Gamma^{(2)})\setminus\calD) \cap \calD_{\bzero} \ne \emptyset$. Since $\calD_{\bzero}$ is an open set,  there exists a point $(p_0,q_0)\in ((\Gamma^{(1)} \cap \Gamma^{(2)})\setminus\bar{\calD}) \cap \calD_{\bzero}$, where $\bar{\calD}$ is the closure of $\calD$. We have $\Phi_{\bzero}(p_0,q_0)<\infty$. 
On the other hand, from the definition of $(p_0,q_0)$, there exists a $\theta\in(\bar{\theta}_1^{(2)},\bar{\theta}_1^{(1)})$ such that $p_0>\theta$ and $q_0>\bar{\zeta}_2(\theta)$, and such $(\theta,\bar{\zeta}_2(\theta))$ can be taken in the point set $\mathbb{D}_0$ since $\mathbb{D}_0$ is dense in $\mathbb{D}$. 
Hence, we have $[\Phi_{\bzero}(p_0,q_0)]_{j,j'}\ge \varphi_{\bc}(p_0,q_0) = \infty$, and this contradicts finiteness of $\Phi_{\bzero}(p_0,q_0)$. 
As a result, we have $((\Gamma^{(1)} \cap \Gamma^{(2)})\setminus\calD) \cap \calD_{\bzero} = \emptyset$ and this completes the proof. 
\end{proof}

%
%
\subsection{Asymptotic decay rates of marginal measures}

Let $(X_1,X_2)$ be a vector of random variables subject to the stationary distribution of  a two-dimensional reflecting random walk. The asymptotic decay rate of the marginal tail distribution in a form $\mathbb{P}(c_1 X_1+c_2 X_2>x)$ has been considered in \cite{Miyazawa11} (also see \cite{Kobayashi14}), where $(c_1,c_2)$ is a direction vector. 
In this subsection, we consider this type of asymptotic decay rate for the occupation measures. 

Let $c_1$ and $c_2$ be mutually prime positive integers. We assume $c_1\le c_2$; in the case of $c_1>c_2$, an analogous result can be obtained. For $k\ge 0$, define an index set $\scrI_k$ as
\[
\scrI_k=\{ l_2\in\mathbb{Z}_+; c_1 l_1 + c_2 l_2 = c_1 k\ \mbox{for some}\ l_1\in\mathbb{Z}_+ \}. 
\]
For $\bx\in\mathbb{Z}_+^2$, the matrix moment generating function $\Phi_{\bx}(c_1 \theta,c_2 \theta)$ is represented as 
\begin{equation}
\Phi_{\bx}(c_1 \theta,c_2 \theta) = \sum_{k=0}^\infty e^{k c_1\theta}  \sum_{l\in\scrI_n} N_{\bx,(k-(c_2/c_1) l,l)}. 
\end{equation}
By the Cauchy-Hadamard theorem, we obtain the following theorem.
\begin{theorem} \label{th:asymptotic_marginal}
For any mutually prime positive integers $c_1$ and $c_2$ such that $c_1\le c_2$ and for every $(\bx,j)\in\mathbb{S}_+$ and $j'\in S_0$, 
\begin{align}
& \limsup_{k\to\infty} \frac{1}{k} \log  \sum_{l\in\scrI_k} \tilde{q}_{(\bx,j),(k-(c_2/c_1) l,l,j')} 
= - \sup_{(c_1 \theta,c_2 \theta)\in\Gamma} c_1 \theta.
\end{align}
In the case where $c_2\le c_1$, an analogous result holds. 
\end{theorem}

%
%

%
\section{Concluding remarks} \label{sec:concluding}

Using our results, we can obtain a lower bound for the asymptotic decay rate of the stationary distribution in a 2d-QBD process. 
Let $\{\tilde{\bY}_n\}=\{(\tilde{X}_{1,n},\tilde{X}_{2,n},\tilde{J}_n)\}$ be a 2d-QBD process on state space $\mathbb{S}_+=\mathbb{Z}_+^2\times S_0$ and assume that the blocks of transition probabilities when $X_{1,n}>0$ and $X_{2,n}>0$ are given by  $A_{i,j},i,j\in\{-1,0,1\}$. 
Assume that $\{\bY_n\}$ is irreducible, aperiodic and positive recurrent and denote by $\bnu=((\nu_{x_1},\nu_{x_2},j); (x_1,x_2,j)\in\mathbb{S}_+)$ the stationary distribution of the 2d-QBD process. Further assume that the blocks $A_{i,j},i,j\in\{-1,0,1\},$ satisfy the property corresponding to Assumptions \ref{as:P_irreducible} and \ref{as:Q_irreducible}. 
For $x_1,x_2\ge 1$ and $j\in S_0$, we have
\begin{align}
\nu_{(x_1,x_2,j)} 
&= \sum_{j',j''\in S_0} \nu_{(0,0,j')} \tilde{p}_{(0,0,j'),(1,1,j'')} \tilde{q}_{(0,0,j''),(x_1-1,x_2-1,j)} \cr
&\quad + \sum_{l\in\{0,1\}}\, \sum_{j',j''\in S_0} \bigl\{ \nu_{(1,0,j')} \tilde{p}_{(1,0,j'),(1+l,1,j'')} \tilde{q}_{(l,0,j''),(x_1-1,x_2-1,j)} \cr
&\qquad\qquad\qquad\qquad + \nu_{(0,1,j')} \tilde{p}_{(0,1,j'),(1,1+l,j'')} \tilde{q}_{(0,l,j''),(x_1-1,x_2-1,j)} \bigr\} \cr
&\quad + \sum_{k=2}^\infty\, \sum_{l\in\{-1,0,1\}}\, \sum_{j',j''\in S_0} \bigl\{ \nu_{(k,0,j')} \tilde{p}_{(k,0,j'),(k+l,1,j'')} \tilde{q}_{(k+l-1,0,j''),(x_1-1,x_2-1,j)} \cr
&\qquad\qquad\qquad\qquad\qquad + \nu_{(0,k,j')} \tilde{p}_{(0,k,j'),(1,k+l,j'')} \tilde{q}_{(0,k+l-1,j''),(x_1-1,x_2-1,j)} \bigr\}, 
\label{eq:nu_tildeq}
\end{align}
where, for $\by,\by'\in\mathbb{S}_+$, $\tilde{p}_{\by,\by'}=\mathbb{P}(\tilde{\bY}_1=\by'\,|\,\tilde{\bY}_0=\by)$ and $\tilde{q}_{\by,\by'}$ is an element of the occupation measures in the corresponding 2d-MMRW, defined by \eqref{eq:tildeq_def}. 
By \eqref{eq:nu_tildeq} and Theorem \ref{th:asymptotic_any_direction}, for any vector $\bc=(c_1,c_2)$ of positive integers and for every $j\in S_0$, a lower bound for the asymptotic decay rate of the stationary distribution in the 2d-QBD process in the direction specified by $\bc$ is given as follows:
\begin{equation}
\liminf_{k\to\infty} \frac{1}{k} \log \nu_{(c_1 k,c_2 k,j)} 
\ge - \sup\{ \langle \bc, \btheta \rangle; \spr(A_{*,*}(\btheta))< 1,\,\btheta\in\mathbb{R}^2 \},
\label{eq:liminf_nu}
\end{equation}
where $A_{*,*}(\btheta)=\sum_{i,j\in\{-1,0,1\}} e^{\langle (i,j),\btheta \rangle} A_{i,j}$. 
Note that an upper bound for the asymptotic decay rate can be obtained by using the convergence domain of the matrix moment generating function for the stationary distribution and an inequality corresponding to \eqref{eq:tildeq_upper}. The convergence domain can be determined by Lemma 3.1 of \cite{Ozawa18} and Corollary \ref{co:domain_Phix}. 


%
%

%
%
\appendix

\section{Proof of Proposition \ref{pr:QBDcp}} \label{sec:proof_QBDcp}

We use the following proposition for proving Proposition \ref{pr:QBDcp}. 
\begin{proposition} \label{pr:block_cp}
Let $C_{-1}$, $C_0$ and $C_1$ be $m\times m$ nonnegative matrices, where $m$ can be countably infinite, and define a matrix function $C_*(\theta)$ as 
\begin{equation}
C_*(\theta) = e^{-\theta} C_{-1} + C_0 + e^{\theta} C_1.
\end{equation}
Assume that, for any $n\in\mathbb{Z}_+$, $C_*(0)^n$ is finite and $C_*(0)$ is irreducible. 
Let $k$ be a positive integer and define a $k\times k$ block matrix $C^{[k]}(\theta)$ as 
\begin{align}
&C^{[k]}(\theta) = 
\begin{pmatrix}
C_0 & C_1 & & & e^{-\theta} C_{-1} \cr
C_{-1} & C_0 & C_1 & & & \cr
& \ddots & \ddots & \ddots & & \cr
& & C_{-1} & C_0 & C_1 \cr
e^{\theta} C_1 & & & C_{-1} & C_0
\end{pmatrix}.
\end{align}
Then, we have $\cp(C_*(\theta))=\cp(C^{[k]}(k \theta))$. 
\end{proposition}

\begin{proof}
First, assume that, for a positive number $\beta$ and a measure $\bu$, $\beta \bu C_*(\theta) \le \bu$, and define a measure $\bu^{[k]}$ as
\[
\bu^{[k]} = 
\begin{pmatrix} 
e^{(k-1)\theta} \bu & e^{(k-2)\theta} \bu & \cdots & e^{\theta} \bu & \bu
\end{pmatrix}.
\]
Then, we have $\beta \bu^{[k]} C^{[k]}(k\theta) \le \bu^{[k]}$ and, by Theorem 6.3 of \cite{Seneta06}, we obtain $\cp(C_*(\theta)) \le \cp(C^{[k]}(k\theta))$. 

Next, assume that, for a positive number $\beta$ and a measure $\bu^{[k]}=\begin{pmatrix} \bu_1 & \bu_2 & \cdots & \bu_k \end{pmatrix}$, $\beta \bu^{[k]} C^{[k]}(k\theta) \le \bu^{[k]}$, and define a measure $\bu$ as 
\[
\bu = e^{-(k-1)\theta} \bu_1 + e^{-(k-2)\theta} \bu_2 + \cdots + e^{-\theta} \bu_{k-1} +\bu_k.
\]
Further, define a nonnegative matrix $V^{[k]}$ as
\[
V^{[k]} = 
\begin{pmatrix}
e^{-(k-1)\theta} I & e^{-(k-2)\theta} I  & \cdots & e^{-\theta} I & I
\end{pmatrix}.
\]
Then, we have $\beta \bu^{[k]} C^{[k]}(k\theta) V^{[k]} = \beta \bu C_*(\theta)$ and $\bu^{[k]} V^{[k]} = \bu$. Hence, we have $\beta \bu C_*(\theta) \le \bu$ and this implies $\cp(C^{[k]}(k\theta)) \le \cp(C_*(\theta))$.
\end{proof}

%
\begin{proof}[Proof of Proposition  \ref{pr:QBDcp}]
For $j\in\{-1,0,1\}$, define a matrix function $B_{*,j}(\theta_1)$ as 
\[
B_{*,j}(\theta_1) = e^{-\theta_1} B_{-1,j} + B_{0,j} + e^{\theta_1} B_{1,j}. 
\]
The matrix function ${}^{\bc}\!A_{*,*}(\theta_1,\theta_2)$ is a $c_1c_2\times c_1c_2$ block matrix and given as
\begin{equation}
{}^{\bc}\!A_{*,*}(\theta_1,\theta_2) = 
\begin{pmatrix}
B_{*,0}(\theta_1) & B_{*,1}(\theta_1) & & & e^{-\theta_2} B_{*,-1}(\theta_1) \cr
B_{*,-1}(\theta_1) & B_{*,0}(\theta_1) & B_{*,1}(\theta_1) & & \cr
& \ddots & \ddots & \ddots & & \cr
& & B_{*,-1}(\theta_1) & B_{*,0}(\theta_1) & B_{*,1}(\theta_1) \cr
e^{\theta_2} B_{*,1}(\theta_1) & & & B_{*,-1}(\theta_1) & B_{*,0}(\theta_1)
\end{pmatrix}.
\label{eq:cAss}
\end{equation}
Define a matrix function $B_{*,*}(\theta_1,\theta_2)$ as
\[
B_{*,*}(\theta_1,\theta_2) = e^{-\theta_2} B_{*,-1}(\theta_1) + B_{*,0}(\theta_1) + e^{\theta_2} B_{*,1}(\theta_1). 
\]
Then, by Proposition \ref{pr:block_cp} and \eqref{eq:cAss}, we have $\cp(B_{*,*}(\theta_1,\theta_2))=\cp({}^{\bc}\!A_{*,*}(\theta_1,c_2 \theta_2))$. 
The matrix function $B_{*,*}(\theta_1,\theta_2)$ is a $c_1\times c_1$ block matrix and given as 
\begin{equation}
B_{*,*}(\theta_1,\theta_2) = 
\begin{pmatrix}
A_{0,*}(\theta_2) & A_{1,*}(\theta_2) & & & e^{-\theta_1} A_{-1,*}(\theta_2) \cr
A_{-1,*}(\theta_2) & A_{0,*}(\theta_2) & A_{1,*}(\theta_2) & & \cr
& \ddots & \ddots & \ddots & & \cr
& & A_{-1,*}(\theta_2) & A_{0,*}(\theta_2) & A_{1,*}(\theta_2) \cr
e^{\theta_1} A_{1,*}(\theta_2) & & & A_{-1,*}(\theta_2) & A_{0,*}(\theta_2)
\end{pmatrix}.
\label{eq:Bss}
\end{equation}
Hence, by Proposition \ref{pr:block_cp}, we have $\cp(A_{*,*}(\theta_1,\theta_2))=\cp(B_{*,*}(c_1\theta_1,\theta_2))$ and this implies 
\[
\cp(A_{*,*}(\theta_1,\theta_2)) 
=\cp(B_{*,*}(c_1\theta_1,\theta_2))
=\cp({}^{\bc}\!A_{*,*}(c_1\theta_1,c_2 \theta_2)). 
\]
\end{proof}

\end{document}